\documentclass[10pt,letterpaper]{article}

\usepackage[letterpaper, top=1in, bottom=1in, left=1in, right=1in, marginparwidth=2.2cm, marginparsep=1mm, includefoot]{geometry}

\usepackage{latexsym,
amsthm,
color,
amssymb,
url,
bm,
cite,
amssymb,
microtype,
amsmath,
amsfonts,
mathtools}

\definecolor{BrilliantRose}{rgb}{1.0, 0.33, 0.64}
\definecolor{denim}{rgb}{0.08, 0.38, 0.74}
\definecolor{amethyst}{rgb}{0.6, 0.4, 0.8}

\usepackage{amssymb}
\usepackage{microtype}
\usepackage{amsmath}
\usepackage{amsfonts}
\usepackage{mathtools}
\usepackage[mathscr]{euscript}
\usepackage{amsthm}
\usepackage{nicefrac}
\usepackage{dsfont}

\usepackage{caption}
\usepackage{setspace}
\usepackage{subcaption}
\usepackage{setspace}

 \usepackage[pdftex, plainpages = false, pdfpagelabels, 
                 bookmarks = true,
                 bookmarksopen = true,
                 bookmarksnumbered = true,
                 breaklinks = true,
                 linktocpage,
                 pagebackref,
                 colorlinks = true,  
                 linkcolor = blue,
                 urlcolor  = blue,
                 citecolor = red,
                 anchorcolor = green,
                 hyperindex = true,
                 hyperfigures
                 ]{hyperref}

\usepackage[inline]{enumitem}
\usepackage[mathscr]{euscript}
\usepackage{amsthm}
\usepackage{nicefrac}
\usepackage{dsfont}
\usepackage{tikz}
\usepackage[ruled,linesnumbered]{algorithm2e}
\usepackage{figures/kbordermatrix}

\usepackage{nicefrac}
\usepackage[mathscr]{euscript}

\RequirePackage[T1]{fontenc}
\RequirePackage[utf8]{inputenc}
\usepackage[english]{babel}
\usepackage{alphabeta}
\usepackage{graphicx}

\usetikzlibrary{math}

\makeatletter
\input{colordvi}

\usepackage{aliascnt}

\hypersetup{
    colorlinks,
    linkcolor={red!75!black},
    citecolor={green!75!black},
    urlcolor={blue!75!black},
}

\newcommand{\cref}[1]{\autoref{#1}}

\long\def\reviewresponseaux#1#2{%
  \expandafter\gdef\csname reviewresponse@#1\endcsname{#2}%
}
\long\def\reviewstore#1#2{%
  \begingroup
    \protected@write\@auxout{}%
      {\string\reviewresponseaux{#1}{\unexpanded{#2}}}%
  \endgroup
}

\long\def\reviewcommentaux#1#2{%
  \expandafter\gdef\csname reviewcomment@#1\endcsname{#2}%
}
\long\def\reviewcommentstore#1#2{%
  \begingroup
    \protected@write\@auxout{}%
      {\string\reviewcommentaux{#1}{\unexpanded{#2}}}%
  \endgroup
}
\newcommand{\reviewcommentget}[1]{%
  \ifcsname reviewcomment@#1\endcsname
    \csname reviewcomment@#1\endcsname
  \fi
}

\newcommand{\reviewresponseget}[1]{%
  \ifcsname reviewresponse@#1\endcsname
    \csname reviewresponse@#1\endcsname
  \fi
}
\newcommand{\reviewresponsemargin}[2]{%
  \if\relax\detokenize{#2}\relax
  \else
    \\[.5ex]\protect\hyperlink{reviewresponsefront-#1}{\reviewresponseprefix\ \reviewresponsestyle{#2}}%
  \fi
}

\definecolor{BrilliantRose}{rgb}{1.0, 0.33, 0.64}
\definecolor{denim}{rgb}{0.08, 0.38, 0.74}
\definecolor{amethyst}{rgb}{0.6, 0.4, 0.8}

\usepackage{todonotes}

\newcommand{\revcommentref}[1]{\hyperref[#1]{\textbf{Comment~\ref*{#1}:}}}
\newcommand{\revissueref}[1]{\hyperref[#1]{\textbf{Issue~\ref*{#1}:}}}
\newcommand{\revminorref}[1]{\hyperref[#1]{\textbf{Minor issue~\ref*{#1}:}}}

\usepackage[normalem]{ulem}
\usepackage{xcolor}

\newcommand{\remove}[1]{}

\newcounter{func}
\newcommand{\funref}[1]{\hyperref[#1]{f_{\ref*{#1}}}} % print a
\tikzset{black node/.style={draw, circle, fill = black, minimum size = 5pt, inner sep = 0pt}}
\tikzset{white node/.style={draw, circlternary_treese, fill = white, minimum size = 5pt, inner sep = 0pt}}
\tikzset{normal/.style = {draw=none, fill = none}}
\tikzset{lean/.style = {draw=none, rectangle, fill = none, minimum size = 0pt, inner sep = 0pt}}
\usetikzlibrary{decorations.pathreplacing}
\usetikzlibrary{arrows.meta}

\usetikzlibrary{shapes}
\tikzset{diam/.style={draw, diamond, fill = black, minimum size = 7pt, inner sep = 0pt}}
\newcommand{\Acal}{\mathcal{A}}

\newcommand{\Ccal}{\mathcal{C}}

\newcommand{\Ecal}{\mathcal{E}}

\newcommand{\Gcal}{\mathcal{G}}

\newcommand{\Ocal}{\mathcal{O}}

\newcommand{\Scal}{\mathcal{S}}

\newcommand{\Nbbb}{\mathbb{N}}

\newcommand{\Sbbb}{\mathbb{S}}

\RequirePackage{stmaryrd}
\usepackage{textcomp}
\DeclareUnicodeCharacter{2286}{\subseteq}
\DeclareUnicodeCharacter{2192}{\ifmmode\to\else\textrightarrow\fi}
\DeclareUnicodeCharacter{2203}{\ensuremath\exists}
\DeclareUnicodeCharacter{183}{\cdot}
\DeclareUnicodeCharacter{2200}{\forall}
\DeclareUnicodeCharacter{2264}{\leq}
\DeclareUnicodeCharacter{2265}{\geq}
\DeclareUnicodeCharacter{8614}{\mathbin{\mapsto}}
\DeclareUnicodeCharacter{8656}{\Leftarrow}
\DeclareUnicodeCharacter{8657}{\Uparrow}
\DeclareUnicodeCharacter{8658}{\Rightarrow}
\DeclareUnicodeCharacter{8659}{\Downarrow}
\DeclareUnicodeCharacter{8669}{\rightsquigarrow}
\newcommand{\eqdef}{\stackrel{{\scriptsize\rm def}}{=}}
\DeclareUnicodeCharacter{8797}{\eqdef}
\DeclareUnicodeCharacter{8870}{\vdash}
\DeclareUnicodeCharacter{8873}{\Vdash}
\DeclareUnicodeCharacter{22A7}{\models}
\DeclareUnicodeCharacter{9121}{\lceil}
\DeclareUnicodeCharacter{9123}{\lfloor}
\DeclareUnicodeCharacter{9124}{\rceil}
\DeclareUnicodeCharacter{2208}{\in}
\DeclareUnicodeCharacter{9126}{\rfloor}
\DeclareUnicodeCharacter{9655}{\triangleright}
\DeclareUnicodeCharacter{9665}{\triangleleft}
\DeclareUnicodeCharacter{9671}{\diamond}
\DeclareUnicodeCharacter{9675}{\circ}
\DeclareUnicodeCharacter{10178}{\bot}
\DeclareUnicodeCharacter{10214}{} % needs stmaryrd
\DeclareUnicodeCharacter{10215}{} % needs stmaryrd
\DeclareUnicodeCharacter{10229}{\longleftarrow}
\DeclareUnicodeCharacter{10230}{\longrightarrow}
\DeclareUnicodeCharacter{10231}{\longleftrightarrow}
\DeclareUnicodeCharacter{10232}{\Longleftarrow}
\DeclareUnicodeCharacter{10233}{\Longrightarrow}
\DeclareUnicodeCharacter{10234}{\Longleftrightarrow}
\DeclareUnicodeCharacter{10236}{\longmapsto}
\DeclareUnicodeCharacter{10238}{\Longmapsto} % needs stmaryrd
\DeclareUnicodeCharacter{10503}{\Mapsto}    % needs stmaryrd
\DeclareUnicodeCharacter{10971}{\mathrel{\not\hspace{-0.2em}\cap}}
\DeclareUnicodeCharacter{65294}{\ldotp}
\DeclareUnicodeCharacter{65372}{\mid}

\definecolor{darkred}{rgb}{.5, 0, 0}
\definecolor{DarkTangerine}{rgb}{1.0, 0.66, 0.07}
\definecolor{darkyellow}{rgb}{.7, .6, 0.0}
\definecolor{CornflowerBlue}{rgb}{0.39, 0.58, 0.93}
\definecolor{DarkGoldenrod}{rgb}{0.72, 0.53, 0.04}
\definecolor{BritishRacingGreen}{rgb}{0.0, 0.26, 0.15}
\definecolor{AO}{rgb}{0.0, 0.5, 0.0}
\definecolor{MidnightBlack}{rgb}{0.1,0.1,.34}
\definecolor{MidnightBlue}{rgb}{0.1,0.1,0.43}
\definecolor{Black}{rgb}{0,0, 0}
\definecolor{Blue}{rgb}{0, 0 ,1}
\definecolor{Red}{rgb}{1, 0 ,0}
\definecolor{White}{rgb}{1, 1, 1}
\definecolor{DeepMagenta}{rgb}{0.8, 0.0, 0.8}
\definecolor{grey}{rgb}{.6, .6, .6}
\definecolor{darkgrey}{rgb}{.33, .33, .33}
\definecolor{Mygreen}{rgb}{.0, .7, .0}
\definecolor{Yellow}{rgb}{.55,.55,0}
\definecolor{Mustard}{rgb}{1.0, 0.86, 0.35}
\definecolor{applegreen}{rgb}{0.55, 0.71, 0.0}
\definecolor{darkturquoise}{rgb}{0.0, 0.81, 0.82}
\definecolor{celestialblue}{rgb}{0.29, 0.59, 0.82}
\definecolor{green_yellow}{rgb}{0.68, 1.0, 0.18}
\definecolor{crimsonglory}{rgb}{0.75, 0.0, 0.2}
\definecolor{darkmagenta}{rgb}{0.30, 0.0, 0.30}
\definecolor{magenta}{rgb}{0.50, 0.0, 0.50}
\definecolor{internationalorange}{rgb}{1.0, 0.31, 0.0}
\definecolor{darkorange}{rgb}{1.0, 0.55, 0.0}
\definecolor{ao}{rgb}{0.0, 0.5, 0.0}
\definecolor{awesome}{rgb}{1.0, 0.13, 0.32}
\definecolor{darkcyan}{rgb}{0.0, 0.50, 0.50}
\definecolor{violet}{rgb}{0.93, 0.51, 0.93}
\definecolor{brown}{rgb}{0.65, 0.16, 0.16}
\definecolor{orange}{rgb}{1.0, 0.65, 0.0}
\definecolor{DarkGreen}{rgb}{0,.5,0}
\definecolor{BostonUniversityRed}{rgb}{0.8, 0.0, 0.0}
\definecolor{chromeyellow}{rgb}{1.0, 0.65, 0.0}

\newcommand{\darkgreen}[1]{{\color{DarkGreen}#1}}

\newcommand{\blue}[1]{{\color{Blue}#1}}
\newcommand{\red}[1]{{\color{Red}#1}}

\newcommand{\darkorange}[1]{{\color{darkorange}#1}}
\newcommand{\violet}[1]{{\color{violet}#1}}

\definecolor{Red}{rgb}{1, 0 ,0}
\definecolor{Blue}{rgb}{0, 0 ,1}

\newtheorem{theorem}{Theorem}[section]

\newaliascnt{question}{theorem}

\aliascntresetthe{question}

\newaliascnt{lemma}{theorem}
\newtheorem{lemma}[lemma]{Lemma}
\aliascntresetthe{lemma}

\newaliascnt{claim}{theorem}
\newtheorem{claim}[claim]{Claim}
\aliascntresetthe{claim}

\newaliascnt{invariant}{theorem}

\aliascntresetthe{invariant}

\newaliascnt{proposition}{theorem}
\newtheorem{proposition}[proposition]{Proposition}
\aliascntresetthe{proposition}

\newaliascnt{observation}{theorem}
\newtheorem{observation}[observation]{Observation}
\aliascntresetthe{observation}

\newaliascnt{corollary}{theorem}
\newtheorem{corollary}[corollary]{Corollary}
\aliascntresetthe{corollary}

\newaliascnt{definition}{theorem}

\aliascntresetthe{definition}

\newaliascnt{conjecture}{theorem}

\aliascntresetthe{conjecture}

\newaliascnt{counterexample}{theorem}

\aliascntresetthe{counterexample}

\newcommand{\hh}{\end{document}}

\newcommand{\p}{{\sf p}}
\newcommand{\sobs}{{\sf sobs}}

\newcommand{\gall}{\mathcal{G}_{{\text{\rm  \textsf{all}}}}}
\newcommand{\Gall}{\mathcal{G}_{{\text{\rm  \textsf{all}}}}}

\newcommand{\hw}{{\sf hw}\xspace}%Hadwiger number
\newcommand{\tw}{{\sf tw}\xspace}%treewdith
\newcommand{\size}{{\sf size}\xspace}%size
\newcommand{\bg}{{\sf bg}\xspace}%biggest grid
\newcommand{\bdim}{{\sf bdim}\xspace}%Bidimensionality
\newcommand{\cupall}{\pmb{\pmb{\bigcup}}}

\newenvironment{cproof}{\proof[Proof of claim]}{\endproof}

\newcommand{\torso}{\ensuremath{\mathsf{torso}}}

\newcommand{\poly}{\text{$\mathsf{poly}$}\xspace}

\newcommand{\eg}{{\sf eg}\xspace}%Euler genus

\usepackage{color}

\title{The Graph Minor Structure Theorem through Bidimensionality\thanks{Emails of authors: 
  \texttt{sedthilk@thilikos.info}, \texttt{wiederrecht@kaist.ac.kr}}}

\author{\bigskip\large Dimitrios M. Thilikos\thanks{LIRMM, Univ Montpellier, CNRS, Montpellier, France.}~$^{,}$\thanks{Supported by the ANR projects ESIGMA (ANR-17-CE23-0010) and  GODASse (ANR-24-CE48-4377),  the French-German Collaboration ANR/DFG Project UTMA (ANR-20-CE92-0027),  by the MEAE and the MESR, via the Franco-Norwegian project PHC
Aurora project n. 51260WL (2024-5), and the France 2030 grant reference number ANR-24-RRII-0002 operated by the
Inria Quadrant Program.}\and\large 
 \and\large 
 Sebastian Wiederrecht\thanks{School of Computing, KAIST, Daejeon, South Korea.}}
\date{}

\begin{document}

\maketitle

\begin{abstract}  
\noindent The \emph{bidimensionality} of a set of vertices $X$ in a graph $G$ is the maximum $k$ for which $G$ contains as a $X$-rooted minor the $(k \times k)$-grid. 
This notion allows for the following version of the Graph Minors Structure Theorem (GMST) that avoids the use of  apices and vortices:
\textsl{$K_k$-minor free graphs are those that admit tree decompositions whose torsos   contain sets of bounded bidimensionality   whose removal yield a graph embeddable in some surface $\Sigma$ of bounded Euler-genus.}
We next fix the target condition by demanding that $\Sigma$ is some particular surface. 
This defines a ``surface extension'' of treewidth, where $\Sigma\mbox{-}\tw(G)$ is the minimum $k$ for which $G$ admits a tree decomposition whose torsos become  embeddable in $\Sigma$
after the removal of a set of bidimensionality at most $k$.   
We identify a finite collection $\mathfrak{D}_{\Sigma}$ of parametric graphs and prove that the minor-exclusion of the graphs in $\mathfrak{D}_{\Sigma}$  determines the behavior of ${\Sigma}\mbox{-}\tw,$ for every surface $\Sigma.$ It follows that the collection $\mathfrak{D}_{\Sigma}$ \textsl{bijectively  corresponds} to the ``surface obstructions'' for $\Sigma,$ i.e., surfaces that are minimally non-contained in $\Sigma.$
Our results are tight in the sense that ${\Sigma}\mbox{-}\tw$ cannot be bounded 
for all parametric graphs in $\mathfrak{D}_{\Sigma}$.
\end{abstract}
\medskip


\newpage
%\thispagestyle{empty}
%\tableofcontents
%\newpage

\pagenumbering{arabic}

\section{Introduction}\label{mentalization}

An underlying fact in many results in structural graph theory is that graphs that do not ``contain'' some  graph can be seen as the result of gluing together simpler graphs of some particular property in a tree-like fashion.
When the containment relation is the minor relation\footnote{We say that a graph $H$ is a \emph{minor} of a graph $G$ if $H$ can be obtained from some subgraph of $G$ by contracting edges.} this property is typically enjoying certain topological characteristics.
Moreover, in some cases, such structural theorems may serve as the combinatorial base for algorithmic applications.

Two celebrated, and highly related, theorems of this type are the \textsl{Grid Theorem} and the \textsl{Graph Minors Structure Theorem} (GMST), proven by Robertson and Seymour in {\cite{robertson1986graph}} and \cite{robertson2003graph} respectively.

The Grid Theorem asserts that every graph which cannot be obtained by gluing together small graphs in a tree-like fashion must contain all planar graphs up to a certain size as minors.
Vice-versa, if some graph $G$ does not contain a fixed planar graph $H$ as a minor, $G$ can indeed be obtained by gluing together small graphs in a tree-like fashion where the notion of ``small'' now depends on the excluded graph $H.$

Taking this insight as the base case, the GMST gives the most general version of a similar statement.
Here, the excluded minor is no longer planar, but a clique of fixed size.
Any graph excluding such a large clique minor can be obtained by gluing together graphs that ``almost'' embed in some surface of bounded Euler-genus in, again, a tree-like fashion.
This is a property that is not enjoyed by large complete graphs.

The name ``Grid Theorem'' stems from the way Robertson and Seymour prove the above result:
They identify a single \textsl{parametric graph}, namely the $(n\times n)$-grid, and show that

\begin{enumerate}
  \item every planar graph is a minor of the $(n\times n)$-grid for some $n\in\mathbb{N},$
  \item the treewidth of the $(n\times n)$-grid depends on $n,$ and
  \item there exists a function $f$ such that for every $n\in\mathbb{N},$ every graph with treewidth more than $f(n)$ contains the $(n\times n)$-grid as a minor.
\end{enumerate}\medskip

So, the grid fulfills two roles:
It acts as a \textsl{universal obstruction} for the parameter treewidth and it is an {\textsl{omnivore}\footnote{We say that 
a parametric graph $\mathscr{G}=\langle \mathscr{G}_{k}\rangle_{k\in \mathbb{N}}$ is 
a (minor) \emph{omnivore} of some graph class $\Gcal$ if every graph $G\in\Gcal$ is a minor 
of $\mathscr{G}_{k}$ for some $k\in\Nbbb$.} for the class of planar graphs; that is, the class of all minors of grids is exactly the class of planar graphs while the grids themselves are planar.
This means that embeddability on the sphere is tightly captured by the decomposability into small pieces, defined by the parameter of treewidth.
\medskip

This paper is motivated by the following question:
\begin{eqnarray}
\begin{minipage}{14cm}
\textsl{Given a fixed surface $\Sigma,$ is there a graph structure parameter $\emph{$\p$}_{\Sigma}$ that expresses the minor-exclusion of $\Sigma$-embeddable graphs in the same way as treewidth does for planar graphs?}
\end{minipage}
\label{constraint}
\end{eqnarray}

\subsection{Our contribution}
\label{recommendation}

Our strategy to resolve the motivating question above follows the three ingredients given for treewidth.
We identify families of grid-like graphs, we call them \textsl{Dyck-grids} (or \textsl{surface-grids} if we are talking about their less streamlined version).
As shown in \cite{thilikos2026excluding}, each Dyck-grid is a parametric graph which is an omnivore for all graphs that embed in the respective surface.
We then proceed by proving that the minor exclusion of the Dyck-grids corresponding to some set of surfaces $\Sbbb$ leads to a ``surface'' generalization of treewidth, namely  ${\Sbbb}\mbox{-}\tw,$ defined 
via tree decompositions where all torsos ``almost embed'' into some surface in $\Sbbb.$
Finally, we prove that, with respect to our notion of ``almost embeddability'', the Dyck-grids themselves cannot be decomposed in this way.
With this, we obtain an asymptotic  characterization of the resulting width parameters ${\Sbbb}\mbox{-}\mathsf{\tw}$ in form of the exclusion of  parametric graphs.

\paragraph{Bidimensionality.} 

To capture what we mean by ``almost embeddable'' we introduce the notion of \textsl{bidimensionality}.
Given a graph $H$ and a set $X\subseteq V(G),$ we say that $H$ is an \emph{$X$-rooted minor} of $G$ (or, simply, an \emph{$X$-minor of $G$}) 
if there is a collection $\mathcal{S}=\{S_{v}\mid v\in V(H)\}$ of pairwise vertex-disjoint {connected}\footnote{A set $X\subseteq V(G)$ is \emph{connected} in $G$ if the induced subgraph $G[X]$ is a connected graph.} subsets of $V(G),$ each containing at least one vertex of $X$ and such that, for every edge $xy\in E(H),$ the set $S_{x}\cup S_{y}$ is connected in $G.$ %
$H$ is a \emph{minor} of $G,$ denoted by $H\leq  G,$ if $H$ is a $V(G)$-minor of $G.$

We say that $X$ is a \emph{modulator} to some \emph{target} graph property $\Gcal$
in a graph $G$ if $G-X\in\Gcal$.
The modulator/target scheme has been used extensively in order to define  graph modification problems (see e.g., \cite{AgrawalKLPRS21elim,BulianD16graph,BulianD17fixe,BougeretJS20bridg,DLindermayrSV20elimi,EibenGHK21,FominGT22param,SridharanPASK21anftp,JansenK021verte,AgrawalKLPRSZ22delet}).
For instance, the \textsl{vertex cover} of a graph is defined as the minimum size 
of a modulator to being edgeless  and the \textsl{feedback vertex set} of 
a graph is defined as the minimum size of a modulator to acyclicity. Clearly, 
the size of the modulator is not the only 
measure one may define on it. 
Several measures, different than the modulator size, have been introduced and studied (see \cite{EibenGHK21,AgrawalKLPRSZ22delet,JansenK021verte}).  
We next introduce a new measure of this kind.

\smallskip

The \emph{bidimensionality} of a vertex set $X$ of a $G,$ denoted 
by $\bdim(G,X)$, is the maximum
$k$ for which $G$ contains a $(k\times k)$-grid as an $X$-minor. Intuitively, 
bidimensionality measures to what extend 
$X$ can be ``spread'' as a two dimensional mess inside the graph $G$.
We say that a set $X$ is a \emph{$k$-bidimensional modulator} of 
 $G$ to some graph class $\mathcal{G}$ if $\bdim(G,X)\leq k$ and  $G-X\in\mathcal{G}.$
In   our results, the target graph class will be the class $\Ecal_{k}$ containing all graphs of Euler-genus at most $k.$ \medskip

As a first step, we display how the GMST can be restated using the modulator-target duality above.
For this, we consider the \emph{Hadwiger number} of a graph, that is the graph parameter $\hw,$  where $\hw(G)$  is the maximum $t$ such that $G$ contains $K_t$ as a minor.
We define the parameter $\hw'$ as follows
\begin{eqnarray}
\begin{minipage}{14cm}
$\hw'(G)$ is the minimum $k$ for which $G$ is in the {clique-sum closure} of the class containing all graphs with a  $k$-bidimensional modulator to the class of   graphs of Euler-genus at most $k$. 
\end{minipage}\label{stereotypically}
\end{eqnarray}

The \textsl{clique-sum closure}, defined formally in \cref{inphasized},  formalizes the ``tree-decomposability'' notion that underlines structural theorems such as the Grid Theorem and the GMST. 
Our first result is 
the following  variant of the GMST that is proved in \cref{beneficial} (\cref{retribution}).
%stating that $\hw$ and $\hw'$ are \textsl{equivalent}:

%

%

%

%

%
\begin{theorem}\label{main_mainl}
There exists a function $f_{\ref{main_mainl}}:\mathbb{N}\to\mathbb{N}$ such that  every $K_{k}$-minor free graph is the clique sum closure of the class of graphs with $f(k)$-bidimensional modulators to the embeddability is surfaces of Euler genus at most $f(k)$. Moreover, it holds that $f_{\ref{main_mainl}}(k)=2^{\poly(k)}$.\footnote{We use $\poly(k)$ as a shortcut of $k^{\Ocal(1)}$.}
\end{theorem}

% (\cref{retribution}), that  is,  we have $\hw(G)\leq  f(\hw'(G))$ and $\hw'(G)\leq  f(\hw(G)).$
%

%
%
%
%
%
%
The above implies that 
for every graph $G$, $\hw'(G)≤f_{\ref{main_mainl}}(\hw(G))$. Also, it easily follows that $\hw(G)=\Ocal((\hw'(G))^2)$ (see \cref{retribution}).
This implies that the two parameters $\hw$ and $\hw'$ are \textsl{equivalent}\footnote{We use this concept also to define equivalence of any pair of parameters $\mathsf{p}$ and $\mathsf{p}'$. See \cref{inphasized} for a formal definition.}
and we denote this fact by  $\hw\sim \hw'.$  
This  provides an alternative   decomposition-based statement of the GMST where no adhesion bounds are imposed and where apices and vortices\footnote{See \cref{inphasized}
for the definition of the adhesion of a tree decomposition and \cref{all_minors} for a  definition of apices and vortices.} are replaced by a modulator measure:
\textsl{the removal of a ``low-bidimensionality'' modulator}.

\paragraph{Dyck-grids.}
Observe that the definition of $\hw'$ in \eqref{stereotypically} is simultaneously minimizing two quantities:
the bidimensionality of the modulator and the Euler-genus of the resulting graph.
The min-max duality of the GMST is linking this entity with the exclusion of a parametric graph as a minor.
In this case, the parametric graph is the complete graph on $k$ vertices. 
\medskip\medskip

But what about fixing the surface?
What should replace cliques as the ``excluded entity'' to provide a min-max duality for some particular surface? 
In the ``bottom'' case of such a theorem we should position the case where the target condition is void. 
%{In this base case the resulting min-max theorem is the Grid Theorem asserting the equivalence between treewidth and the exclusion of grids (expressing planarity).}}

%
%
%
%
%
%
\begin{figure}[ht]
  \begin{center}
  \scalebox{.9}{\includegraphics{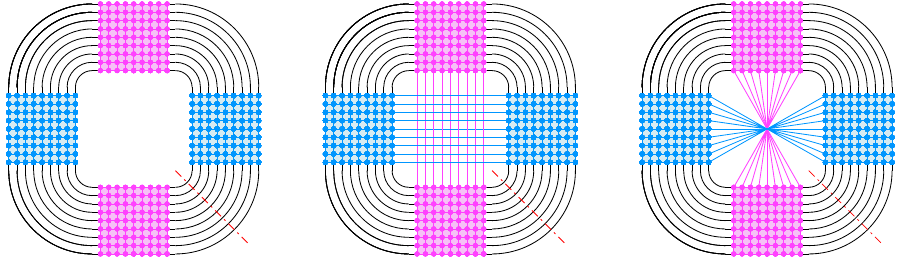}}
  \end{center}
    \caption{The annulus grid $\mathscr{A}_{9},$ the handle grid $\mathscr{H}_{9}$ and the crosscap grid $\mathscr{C}_{9}$ in order from left to right. {Notice that both $\mathscr{H}_{9}$ and $\mathscr{C}_{9}$ contain two 
    $(18\times 9)$ grids as vertex-disjoint subgraphs (depicted in different colors). Moreover, $\mathscr{C}_{9}$ contains a $(18\times 18)$-grid as a spanning subgraph}.}
  \label{supplanted}
\end{figure}

To present our main result and give an answer to the question above, we define three parametric graphs $\mathscr{A}=\langle \mathscr{A}_{k}\rangle_{k\in \mathbb{N}},$ $\mathscr{H}=\langle \mathscr{H}_{k}\rangle_{k\in \mathbb{N}},$  and  $\mathscr{C}=\langle \mathscr{C}_{k}\rangle_{k\in \mathbb{N}}$ as follows: 

An $(n \times  m)$-cylindrical grid is the Cartesian product of a cycle on $n$ vertices and a path on $m$ vertices.
The  \emph{annulus grid} $\mathscr{A}_{k}$ is the $(4k \times k)$-cylindrical grid depicted in the left of \cref{supplanted}.
The \emph{handle grid} $\mathscr{H}_{k}$ (resp. \emph{crosscap grid} $\mathscr{C}_{k}$) is obtained adding in $\mathscr{A}_{k}$ edges as indicated in the middle (resp. right) part of \cref{supplanted}. We refer to the added edges as \emph{transactions} of the handle grid $\mathscr{H}_{k}$ or the crosscap grid $\mathscr{C}_{k}.$

Let now $h\in\mathbb{N}$ and $c\in[0,2].$
We define the parametric graph $\mathscr{D}^{(\mathsf{h},\mathsf{c})}=\langle \mathscr{D}_{k}^{(\mathsf{h},\mathsf{c})}\rangle_{k\in \mathbb{N}}$ by taking one copy of $\mathscr{A}_{k},$ $h$ copies of $\mathscr{H}_{k},$ and  $c\in[0,2]$ copies of $\mathscr{C}_{k},$ then ``cut'' them along the dotted \textcolor{BostonUniversityRed}{red} line, as in \cref{supplanted}, and join them together in the cyclic order $\mathscr{A}_{k},\mathscr{H}_{k},\ldots,\mathscr{H}_{k},\mathscr{C}_{k},\ldots,\mathscr{C}_{k},$ as indicated in \cref{perniciously}.
We call the graph $\mathscr{D}_{k}^{(\mathsf{h},\mathsf{c})}$ the \emph{Dyck-grid} of \emph{order} $k$ \emph{with} $h$ \emph{handles and} $c$ \emph{crosscaps}.
For technical reasons, we make the convention that $\mathscr{D}^{(-1,2)}_k=\mathscr{D}^{(0,0)}_k.$

Let $g\in\mathbb{N}\cup\{-1\}$ be an integer.
If $g\geq  0,$ we define $\Ecal_{g}$ as the class of graphs that are embeddable in some surface $\Sigma$ of Euler-genus at most $g.$
We also set $\Ecal_{-1}$ to be the class containing only the empty graph and we use $\mathcal{G}_{\text{all}}$ to denote the class of all graphs. 
We define the graph parameter $g\mbox{-}\tw\colon\mathcal{G_{\text{all}}}\to\mathbb{N}$ as follows
\begin{eqnarray}
\begin{minipage}{14cm}
$g\mbox{-}\tw(G)$ is the minimum $k$ for which $G$ is in the clique-sum closure of the graphs that have a $k$-bidimensional modulator to $\Ecal_{g}.$
\end{minipage}\label{physiognomy}
\end{eqnarray}

\begin{figure}[ht]
  \begin{center}
  \scalebox{0.88}{\includegraphics{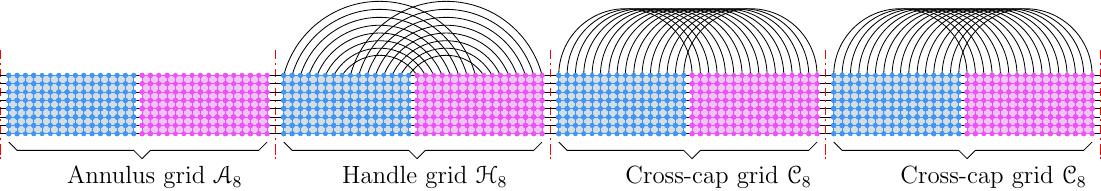}}
  \end{center}
    \caption{The Dyck-grid  of order $8$ with one handle and two crosscaps, i.e., the graph $\mathscr{D}_{8}^{1,2}.$}
  \label{perniciously}
\end{figure}

We stress that the only difference between \eqref{stereotypically} and \eqref{physiognomy}
is that, in  $\eqref{physiognomy},$ the target property is fixed to $\Ecal_{g}.$ Notice that $g\mbox{-}\tw(G)$ can be seen as a parametric extension of surface embeddability in the sense that if $g\mbox{-}\tw(G)\leq k$, then $G$ can be decomposed to graphs that are close ``up to some $k$-bidimensional modulator'' to being embeddable in a surface of Euler-genus $g$.
Using the Grid Theorem, it is easy to verify 
that, in the base case where $g=-1,$  $(-1)\mbox{-}\tw(G)$ is equivalent   to the parameter of treewidth.

Our next step is to define the graph parameter ${g}\mbox{-}\bg$ for $g\in\mathbb{N}$ as follows

\begin{equation}
{g}\mbox{-}\bg(G)\coloneqq
    \begin{cases}
    \max\{k\mid \mathscr{D}^{(0,0)}_{k}\leq  G\} &  \mbox{if $g=0$} \\
     \max\{k\mid {\mathscr{D}^{(\mathsf{h}+1,0)}_{k}\leq  G \mbox{~or~} } \mathscr{D}^{(\mathsf{h},1)}_{k}\leq  G\}   &  \mbox{if $g=2\mathsf{h}+1\geq  1$}\\
    \max\{k\mid \mathscr{D}^{(\mathsf{h}+1,0)}_{k}\leq  G \mbox{~or~} \mathscr{D}^{(\mathsf{h},2)}_{k}\leq  G \} & \mbox{if $g=2\mathsf{h}+2\geq  2$}   \end{cases}\label{commandments} 
\end{equation}

Notice that, in case $g=0$,  ${g}\mbox{-}\bg(G)$ is the maximum $k$ for which 
$G$ contains the annulus $\mathscr{A}_{k}$ as a minor which is equivalent 
to the size of the \textsl{biggest grid} minor of $G$. This motivated the choice of  $\bg$ for the name of this parameter.

We are now ready to state our main result that is the following min-max 
equivalence between the parameters defined in \eqref{physiognomy} and \eqref{commandments}.

\begin{theorem}\label{consoudation}
For every $g\in\mathbb{N}\cup\{-1\},$ the parameters ${g}\mbox{-}\tw$ and ${(g+1)}\mbox{-}\bg$ are  equivalent.
That is, ${g}\mbox{-}\tw\sim{(g+1)}\mbox{-}\bg.$
\end{theorem}

\cref{consoudation} can be seen as an infinite hierarchy of parametric dualities 
starting from the Grid Theorem, that is the base case where $g=-1$, and moving towards the GMST where no fixed bound on $g$ is imposed.

% A contribution of \cref{consoudation} to the general theory of Graph Minors is that it makes the dependencies on the size $|V(H)|$ of the excluded minor $H$, and its Euler-genus visible and utilizable.
% As a consequence of our combinatorial results we are able to deduce a number of improvements on deep algorithmic results in the realm of $H$-minor-free graphs.
% In particular, \cref{consoudation} allows to replace the dependency on the size of $H$ in the multiplicative constants of the linear grid theorem in $H$-minor-free graphs \cite{Kawarabayashi2020LinearMinMax}, the constant-factor polynomial-time approximation algorithm for treewidth of $H$-minor-free graphs \cite{demaine2005algorithmic}, and the balanced separator theorem for $H$-minor-free graphs \cite{MR3024789}, by the Euler-genus of $H$.

\subsection{Surfaces and their obstructions}
\label{undergoing}

\cref{consoudation} is formulated in terms of the Euler-genus, so it does not give an answer to our initial question in \eqref{constraint} on the parameter $\mathsf{p}_{\Sigma}.$
{In fact}, we obtain a much more general result that allows us to answer this question precisely.
To explain our result in full generality, we shift our focus to (families of) surfaces.

\paragraph{Surfaces.}
Given a pair $(\mathsf{h},\mathsf{c})\in\mathbb{N}\times[0,2]$ we define $\Sigma^{(\mathsf{h},\mathsf{c})}$ to be the two-dimensional surface without boundary created from the sphere by adding $h$ handles and $c$ crosscaps.
If $c=0,$ the surface $\Sigma^{(\mathsf{h},\mathsf{c})}$ is an \emph{orientable} surface, otherwise it is a \emph{non-orientable} surface.
By Dyck's theorem \cite{Dyck1888Beitrage,Francis99ConwayZIP},  two crosscaps are equivalent to a handle in the presence of a (third) crosscap. 
This implies that the notation $\Sigma^{(\mathsf{h},\mathsf{c})}$  is sufficient to denote all  two-dimensional surfaces without boundary. 
The \emph{Euler-genus}, denoted by $\eg(\Sigma),$ of the surface $\Sigma=\Sigma^{(\mathsf{h},\mathsf{c})}$ is $2h+c.$
Notice that if a surface has odd Euler-genus $g=2t+1$ then it is always homeomorphic to the non-orientable surface $\Sigma^{(t,1)},$ while there are two surfaces, up to homeomorphism, of even, say $2t,$ Euler-genus: 
the orientable one, namely $\Sigma^{(t,0)},$ and the non-orientable one, that is $\Sigma^{(t-1,2)}.$

%of Gavoille and Hilaire.

%\begin{proposition}[\cite{gavoille2023minor}]\label{territorial}
%For every pair $(\mathsf{h},\mathsf{c})\in\mathbb{N}\times[0,2],$ a graph is embeddable in $\Sigma^{(\mathsf{h},\mathsf{c})}$ if and only if it is a minor of $\mathscr{D}^{(\mathsf{h},\mathsf{c})}_{k},$ for some $k\in\mathbb{N}.$

We  say that a surface $\Sigma$ is \emph{contained} in some surface $\Sigma'$ if $\Sigma'$ can be obtained from $\Sigma$ by adding handles and crosscaps.
We denote this containment relation by $\Sigma'\preceq \Sigma.$ We say that two surfaces $\Sigma_{1}$ and $\Sigma_{2}$ are \emph{non-comparable}
if neither $\Sigma_{1}\preceq\Sigma_{2}$ nor $\Sigma_{2}\preceq\Sigma_{1}$.
We use $\Sigma^{\varnothing}$ for the empty surface and, as a convention, we agree that $\Sigma^{\varnothing}$ is contained in all other surfaces.
Clearly, only the empty graph is embeddable in $\Sigma^{\varnothing}.$
We also fix the Euler-genus of the empty surface $\Sigma^{\varnothing}$ to be $-1.$

Let $\Sbbb$ be a set of surfaces. We say that $\Sbbb$ is {\em closed under containment}, or simply \emph{closed}, if for every $\Sigma\in\Sbbb,$ every surface contained in $\Sigma$ also belongs to $\Sbbb.$ We also define $\eg(\Sbbb)=\max\{\eg(\Sigma)\mid \sigma\in\Sbbb\}$.
\medskip

\paragraph{Surfaces and Dyck-grids.}
Based in a recent result of Gavoille and Hilaire \cite{gavoille2023minor}, the following was proven in  \cite{thilikos2026excluding}
and reveals the correspondence between surfaces and Dyck graphs.

\begin{proposition}\label{territorial}
There exists a function $f_{\ref{territorial}}:\Nbbb^3\to \Nbbb$ such that 
for every pair $(\mathsf{h},\mathsf{c})\in\mathbb{N}\times[0,2],$ a graph $K$ on $n$ vertices is embeddable in $\Sigma^{(\mathsf{h},\mathsf{c})}$ if and only if it is a minor of $\mathscr{D}^{(\mathsf{h},\mathsf{c})}_{k},$ for some $k=f_{\ref{territorial}}(\mathsf{h},\mathsf{c},n)$.
Moreover, $f_{\ref{territorial}}(\mathsf{h},\mathsf{c},n)\in 2^{\mathcal{O}(\mathsf{h}+\mathsf{c})}n^2$.
\end{proposition}

The above indicates that  Dyck-grids can be seen as ``omnivores'' of the graphs embeddable in surfaces: if a graph contains a large enough Dyck-grid for some surface $\Sigma$ as a minor, it also contains every graph embeddable in $\Sigma$ up to a fixed number of vertices as a minor. \medskip

\paragraph{Surface obstructions.}
Let us fix some finite and closed set $\Sbbb$ of surfaces.
The \emph{surface obstruction set} of $\Sbbb$, denoted by $\mathsf{sobs}(\Sbbb),$ is the set of all $\preceq$-minimal surfaces that do not belong to $\Sbbb.$ In other words,
\begin{eqnarray}
\begin{minipage}{12cm}
$\mathsf{sobs}(\Sbbb)\coloneqq \{\Sigma\mid \mbox{$\Sigma$ is a surface where $\Sigma\not\in \Sbbb$ and such that every surface}$\\
$\mbox{\hspace{17mm} that is different from $\Sigma$ and contained in $\Sigma$ belongs to $\Sbbb$}\}.$
\end{minipage}
\label{vivisection}
\end{eqnarray}

Clearly, $\sobs(\Sbbb)$ contains non-comparable surfaces.
For some examples, notice that 
$\mathsf{sobs}(\emptyset)=\{\Sigma^{\varnothing}\},$ 
$\mathsf{sobs}(\{\Sigma^{\varnothing}\})=\{\Sigma^{0,0}\},$  and 
$\mathsf{sobs}(\{\Sigma^{\varnothing},\Sigma^{(0,0)}\})=\{\Sigma^{(1,0)},\Sigma^{(0,1)}\}$.

See \cref{sovereignty} for the lattice corresponding to the surface containment relation and some indicative examples of finite and  closed sets of surfaces and their obstruction sets.
It is important to observe that every orientable surface $\Sigma$ is contained in some non-orientable surface, while the opposite is not correct.
By $\Ecal_{\Sbbb}$ we denote the class of graphs that are embeddable in some surface in $\Sbbb.$
We stress that $\Ecal_{\{\Sigma^{\varnothing}\}}$ is the graph class containing only the empty graph, while $\Ecal_{\emptyset}=\emptyset.$
Notice that there is a one-to-one correspondence between surfaces and Dyck-grids.
For this, given that $\Sigma=\Sigma^{(\mathsf{h},\mathsf{c})},$ we define $\mathscr{D}^{\Sigma}=\langle\mathscr{D}^{\Sigma}_{k}\rangle_{k\in\mathbb{N}}$ as the parametric Dyck-grid $\mathscr{D}^{(\mathsf{h},\mathsf{c})}=\langle\mathscr{D}^{(\mathsf{h},\mathsf{c})}_{k}\rangle_{k\in\mathbb{N}}.$
This permits us to see $\mathscr{D}^{\Sigma}_{k}$ as a parameterization capturing all graphs of $\Ecal_{\{\Sigma\}},$ via minor-containment, as indicated in \cref{territorial}.
For completeness, we define $\mathscr{D}^{\Sigma_{\varnothing}}$ as the sequence containing only empty graphs.

\begin{figure}[ht]
\begin{center}
 \scalebox{.81}{\includegraphics{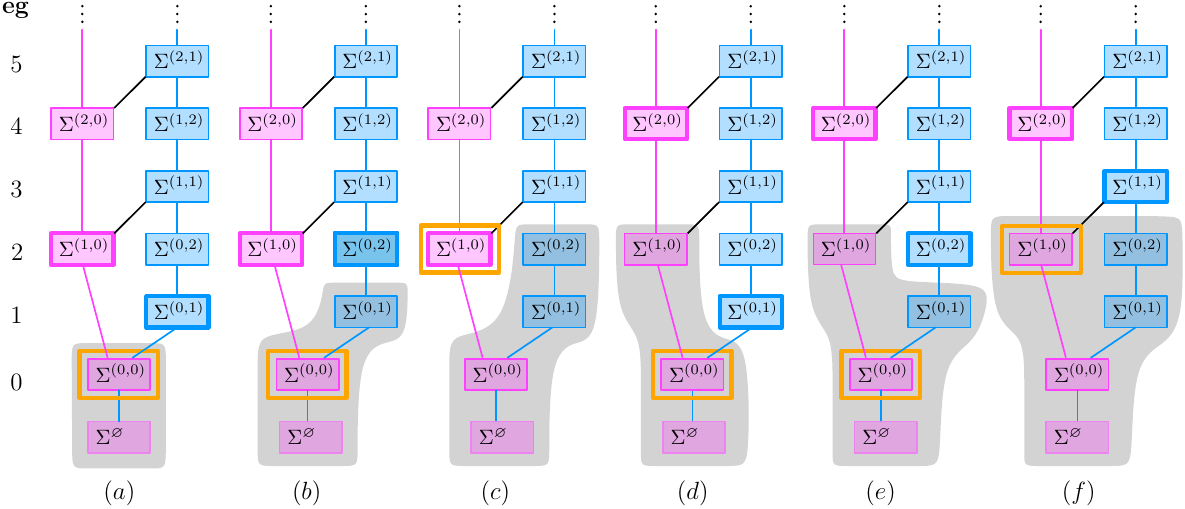}}
\end{center}
  \caption{The lattice of surface containment, along with some indicative closed sets of surfaces $\Sbbb$ (in grey) and their surface obstructions $\mathsf{sobs}(\Sbbb)$ (depicted in bold frames).
  Notice that, in case (c),  $\Sbbb$ is the set of surfaces contained in the Klein bottle $\Sigma^{(0,2)}$ and its the surface obstruction set
  contains only the torus, i.e.,  $\mathsf{sobs}(\Sbbb)=\{\Sigma^{(1,0)}\}.$
For each set of surfaces, the orange box indicates the corresponding prevalent surface.}
\label{sovereignty}
\end{figure}

\paragraph{Surface-based extensions of treewidth.}
We are now ready to state our result in its most general and compact form.
We define the parameter ${\Sbbb}\mbox{-}\tw \colon \mathcal{G_{\text{all}}}\to\mathbb{N}$ such that
\begin{eqnarray}
{\Sbbb}\mbox{-}\tw(G)\coloneqq\min\{k\mid \mbox{$G$ belongs in the clique-sum closure of the graphs}\label{attempting} \\
\mbox{with  a $k$-bidimensional modulator to $\Ecal_{\Sbbb}\}$}.\!\!\!\nonumber
\end{eqnarray} 
Notice that the above definition can be seen as a more refined version of \eqref{stereotypically}.
Again, as a consequence of the  Grid Theorem, it is easy to see
that ${\{\Sigma^{\varnothing}\}}\mbox{-}\tw$
is equivalent to treewidth. Also, as a convention, we agree that ${\emptyset}\mbox{-}\tw = \infty.$ 
Let $\Sbbb'$ be a set of surfaces which is not necessarily closed.
Next we define the graph parameter ${\Sbbb'}\mbox{-}\bg \colon \mathcal{G_{\text{all}}}\to\mathbb{N}$ such that 
\begin{eqnarray}
\Sbbb'\mbox{-}\bg(G) & \coloneqq & \max\{k\mid \mbox{$G$ contains $\mathscr{D}^{\Sigma}_k$ as a minor for some $\Sigma\in\Sbbb'$}\}.\label{travestied}
\end{eqnarray}
We also agree that ${\{\Sigma^{\varnothing}\}}\mbox{-}\bg=\infty,$ while ${\emptyset}\mbox{-}\bg$ is the empty function.\footnote{The empty function is the function corresponding to the empty set, i.e., is the empty relation.}
Similar to before one may observe that \eqref{travestied}
is a more refined version of \eqref{commandments}.
In the following, we are mostly concerned with the case where $\Sbbb'=\mathsf{sobs}(\Sbbb).$ 

The main result of this paper relates the three concepts defined in \eqref{vivisection}, \eqref{attempting}, and \eqref{travestied} as follows.

\begin{theorem}\label{unrelatedly}
For every  finite and closed set of surfaces $\Sbbb,$ the parameters ${\Sbbb}\mbox{-}\mathsf{\tw}$ and ${\mathsf{sobs}(\Sbbb)}\mbox{-}\bg$ are equivalent, i.e., $\Sbbb\mbox{-}\mathsf{\tw}\sim \mathsf{sobs}(\Sbbb)\mbox{-}\bg.$
\end{theorem}

The equivalence above indicates that the min-max duality between surface versions of treewidth and biggest grid is bijectively mapped to the surface obstruction-duality.
\smallskip

In the context of \cref{unrelatedly}, we may view the parametric family $\mathfrak{D}_{\mathsf{sobs}(\Sbbb)}\coloneqq \{ \mathscr{D}^{\Sigma}\mid \Sigma\in\mathsf{sobs}(\Sbbb)\}$ as a \textsl{universal obstruction} for ${\Sbbb}\mbox{-}\mathsf{\tw},$ in the sense that the behavior of the parameter ${\Sbbb}\mbox{-}\mathsf{\tw}$ is determined, up to equivalence, by excluding some graph from each of the parametric graphs in $\mathfrak{D}_{\mathsf{sobs}(\Sbbb)}$  as a minor (see \cref{inphasized} for the formal definition of universal obstructions).

A graph parameter $\mathsf{p}\colon\mathcal{G}_{\text{all}}\to\mathbb{N}$ is \emph{minor-monotone} if for every graph $G$ and every minor $H$ of $G$ we have $\mathsf{p}(H)\leq  \mathsf{p}(G).$
We say that a graph minor-monotone parameter $\p\colon\mathcal{G}_{\text{all}}\to\mathbb{N}$ \emph{is bounded in $\mathcal{G}$} if the set $\{\p(G)\mid G\in \mathcal{G}\}$ is finite.
By \cref{unrelatedly}, $\mathsf{sobs}(\Sbbb)$ precisely describes the asymptotic behavior of ${\Sbbb}\mbox{-}\tw$ in minor-closed graph classes, as indicated by the following result (which, essentially, is a restatement of {\cref{unrelatedly}}).

\begin{theorem}\label{ilemigoddesscs}
Let $\Sbbb$ be a finite and closed set of surfaces. 
Given some minor-closed graph class  $\mathcal{G},$ the parameter ${\Sbbb}\mbox{-}\tw$ is bounded in $\mathcal{G}$ if and only if, for every $\Sigma\in \mathsf{sobs}(\Sbbb),$ $\mathcal{G}$ excludes some graph embeddable in $\Sigma$.
\end{theorem}

\begin{proof}
Suppose that, for every $\Sigma\in \mathsf{sobs}(\Sbbb),$ the class $\mathcal{G}$ excludes some graph $H_{\Sigma}$ that is embeddable in $\Sigma.$
Because of  \cref{territorial},  there is some $k$ such that $\mathcal{G}$ excludes the graph $\mathscr{D}^{\Sigma}_{k},$ for every $\Sigma\in \mathsf{sobs}(\Sbbb).$
Because of \eqref{travestied}, this implies that $\mathsf{sobs}(\Sbbb)\mbox{-}\bg(G)\leq k$.
Therefore, by \cref{unrelatedly}, for every graph $G$ in $\mathcal{G},$ it holds that $\Sbbb\mbox{-}\mathsf{\tw}(G)\leq f(k)$, for some function $f\colon\mathbb{N}\to\mathbb{N}$.
Hence, ${\Sbbb}\mbox{-}\tw$ is bounded in $\mathcal{G}.$
The inverse direction follows  the same steps but the in opposite direction.
\end{proof}

\paragraph{Derivation of \cref{consoudation}.}
By the definition of the Euler-genus of a surface, it follows that ${g}\mbox{-}\tw=\Sbbb_{g}\mbox{-}\tw$ where
\begin{equation*}
\Sbbb_{g}\coloneqq 
    \begin{cases}
    \{\Sigma^\varnothing\} &  \mbox{if $g=-1$}\\
    \{\Sigma\mid  \Sigma\preceq \Sigma^{(\mathsf{h},0)} \mbox{~or~} \Sigma\preceq \Sigma^{(\mathsf{h}-1,2)}\}  &  \mbox{if $g=2\mathsf{h}\geq  0$} \\
    \{\Sigma\mid  \Sigma\preceq \Sigma^{(\mathsf{h},1)}\}   & \mbox{if $g=2\mathsf{h}+1\geq  1$}   \end{cases},
\end{equation*}
where, as we did for Dyck-grids, we make the convention that $\Sigma^{(-1,2)}=\Sigma^{(0,0)}.$
In other words, $\Sbbb_{g}$ is the set of all surfaces of Euler genus $g$, agreeing that  $\Sbbb_{-1}$ contains only the ``empty'' surface $\Sigma^\varnothing$.

Using  the definition of the surface containment relation, it is straightforward to verify the following (see for instance the cases (a), (b), and (f) in \cref{sovereignty}, corresponding to $g=0,1,$ and $2$ respectively):
\begin{equation*}
\mathsf{sobs}(\Sbbb_{g})=
    \begin{cases}
    \{\Sigma^{(0,0)}\} &  \mbox{if $g=-1$}  \\
     \{\Sigma^{(t+1,0)},\Sigma^{(t,1)}\}   &  \mbox{if $g=2t\geq  0$}\\
     \{\Sigma^{(t+1,0)}, \Sigma^{(t,2)}\}  & \mbox{if $g=2t+1\geq  1.$}  \end{cases}
\end{equation*}
Recall that ${\mathsf{sobs}(\Sbbb_{g})}\mbox{-}\bg$ is the maximum $k$ such that $G$ contains the $k$-th member of one of the parametric graphs in $\mathfrak{D}_{\mathsf{sobs}(\Sbbb_{g})}$ as a minor.
From the the definition of ${(g+1)}\mbox{-}\bg$  (by replacing $g$ by $g+1$ in \eqref{commandments}), it follows that ${(g+1)}\mbox{-}\bg={\mathsf{sobs}(\Sbbb_{g})}\mbox{-}\bg.$
Therefore, \cref{consoudation} follows directly from \cref{unrelatedly}, as ${g}\mbox{-}\tw={\Sbbb_g}\mbox{-}\mathsf{\tw}\sim {\mathsf{sobs}(\Sbbb_g)}\mbox{-}\bg={(g+1)}\mbox{-}\bg.$

\paragraph{Prevalent surfaces.}
Let $\Sbbb$ be a finite and closed set of surfaces.
We define the \emph{prevalent surface} of $\mathbb{{S}},$ denoted by $\Sigma_{\Sbbb},$ as the maximal element with respect to the surface containment relation
of the set   
\begin{eqnarray}
\{\Sigma'\mid \text{for all } \Sigma\in \mathsf{sobs}(\Sbbb): \Sigma'\preceq \Sigma\}.\label{sfuibd}
\end{eqnarray}
\begin{observation}
The set in \eqref{sfuibd} has only one maximal element and this element  
  $\Sigma_{\Sbbb}$ is always an orientable surface.
\end{observation}
In \cref{sovereignty}, for each choice of $\Sbbb,$ the prevalent surface is surrounded by an orange  box.
For the most simple cases, it holds that $\Sigma_{\emptyset}=\{\Sigma^{\varnothing}\}$, $\Sigma_{\{\Sigma^{\varnothing}\}}=\{\Sigma^{(0,0)}\}$, and $\Sigma_{\{\Sigma^{\varnothing},\Sigma^{(0,0)}\}}=\Sigma^{(0,0)}$.
One may observe the following (see \cref{sovereignty}).

\begin{observation}
Let $\Sbbb$ be a finite and closed set of surfaces and let $g\coloneqq \max\{\eg(\Sigma) \mid \Sigma\in \Sbbb\}.$
The surface obstruction set $\mathsf{sobs}(\Sbbb)$ has exactly one obstruction if and only if one of the following conditions holds:
\begin{itemize}
\item    all surfaces of Euler-genus $g$ in $\Sbbb$ are non-orientable and $g$ is even and positive, or
\item  $\Sbbb\subseteq \{\Sigma^{\varnothing}\}.$
\end{itemize}
Moreover,  $\Sigma_{\Sbbb}\in\mathsf{sobs}(\Sbbb)$ if and only if $\mathsf{sobs}(\Sbbb)=\{\Sigma_{\Sbbb}\}.$
\end{observation}

The next result indicates  that all graphs embedded in the prevalent surface of $\Sbbb$ are contained as minors of a graph $G,$ with ${\Sbbb}\mbox{-}\tw(G)$ large enough.
\begin{theorem}
\label{farsighted}
For every finite and closed set of surfaces $\Sbbb$ there exists a function $f_{\Sbbb}\colon \mathbb{N}\to\mathbb{N}$ such that for every $t\in\mathbb{N},$ every graph $G$ with ${\Sbbb}\mbox{-}\tw(G)\geq  f_{\Sbbb}(t)$ contains every graph on at most $t$ vertices that is embeddable in the prevalent surface $\Sigma_{\Sbbb}$ as a minor.
\end{theorem}

\begin{proof}
Let $H$ be a graph embeddable in the prevalent surface $\Sigma_{\Sbbb}$ and let $|H|=t.$
By the definition of $\Sigma_{\Sbbb},$  $\Sigma_{\Sbbb}$ is contained in every surface $\Sigma'\in\mathsf{sobs}(\Sbbb),$ therefore $H$ is also embeddable in $\Sigma'.$  By \cref{territorial}, there is some $k_{t}$ such that $H$ is a minor  of $\mathscr{D}^{\Sigma'}_{k_t}.$ 
If now a graph $G$ is $H$-minor free, then it is also $\mathscr{D}^{\Sigma'}_{k_t}$-minor free, therefore, from \eqref{travestied}, $\mathsf{sobs}(\Sbbb)\mbox{-}\bg(G)\leq k_{t}.$ This, by \cref{unrelatedly}, implies that 
$\Sbbb\mbox{-}\mathsf{\tw}(G)\leq f(k_t)$ (for some function 
$f\colon\mathbb{N}\to\mathbb{N}$). 
\end{proof}

In light of our motivating question \ref{constraint}, \cref{farsighted} provides a slightly different perspective.
That is, let $H$ be some graph and $\mathbb{S}_H$ be a finite and closed set of surfaces such that the prevalent surface of $\mathbb{S}_H$ is a surface of smallest Euler-genus where $H$ embeds.
Then $\mathbb{S}_H\text{-}\tw$ is bounded on all $H$-minor-free graphs.

\subsection{The Graph Minors Structure Theorem for fixed surfaces}
\label{labyrinths}

Let us now go back to the original statement of the GMST. For this, we define the parameter $\hw''\colon\mathcal{G}_{\text{all}}\to\mathbb{N}$ where 
\begin{eqnarray}
\begin{minipage}{14cm}
$\hw''(G)$ is the minimum $k$ for which $G$ is in the clique-sum closure of  the graphs that contain a set of at most $k$ vertices whose removal yields a graph admitting a $\Sigma$-decomposition, of breadth at most $k$ and depth at most $k,$ where $\Sigma$ is a surface of Euler-genus at most $k.$
\end{minipage}\label{differences}
\end{eqnarray}

The precise definitions of the concepts of ``$\Sigma$-decomposition'', ``breadth'', and ``depth'' are given in \cref{all_minors}.
Intuitively, a graph $H$ has a $\Sigma$-decomposition with $k$ vortices if $H=H^{(0)}\cup H^{(1)}\cup\cdots\cup H^{(k)}$ where $H^{(0)}$ is a graph embedded in $\Sigma$ and each vortex $H^{(i)}$ is a graph ``attached around'' the vertices of some face of the embedding of $H^{(0)}.$ 
The GMST implies that $\hw''\sim \hw$ (see \cref{retribution}).
As we already mentioned, we prove that $\hw''$ can be replaced by $\hw'$ in this relation.
Our results have also their counterpart in the setting of $\hw''.$
For this we define, given a finite and  closed set of surfaces $\Sbbb,$ the parameter ${\Sbbb}\mbox{-}\hw\colon\mathcal{G}_{\text{all}}\to\mathbb{N}$ where
\begin{eqnarray}
\begin{minipage}{14cm}
$\Sbbb\mbox{-}\hw(G)$ is the minimum $k$ for which $G$ is in the clique-sum closure of the graphs that contain a set of at most $k$ vertices whose removal yields a graph admitting a $\Sigma$-decomposition, of breadth at most $k$ and depth at most $k,$ where $\Sigma\in\Sbbb.$
\end{minipage}\label{eradication}
\end{eqnarray}

Our results imply that the above surface-oriented variant of $\hw''$
is equivalent to the corresponding surface variants of $\hw$ and $\hw'.$

\begin{theorem}
\label{annihilation}
For every  finite and   closed set of surfaces $\Sbbb,$
it holds that the parameters ${\Sbbb}\mbox{-}\mathsf{\tw},$ ${\Sbbb}\mbox{-}\mathsf{\hw}$ and ${\mathsf{sobs}(\Sbbb)}\mbox{-}\bg$ are equivalent, i.e., $\Sbbb\mbox{-}\mathsf{\tw}\sim{\Sbbb\mbox{-}\mathsf{\hw}}\sim {\mathsf{sobs}(\Sbbb)\mbox{-}\bg}.$
\end{theorem}

\cref{annihilation} is stated in a more extensive form (as \cref{abominations}) in \cref{beneficial}, where we make further assumptions on the way the vortices are intersecting the adhesions of the clique-sums.
Such assumptions do not (asymptotically) alter the parameters in question and are useful in  applications of the decomposition theorems emerging from \cref{annihilation}.
\smallskip

We finally wish to mention that, based on the results of \cite{thilikos2026excluding},  all our proofs are algorithmic 
and can be resumed to the following: 

\begin{theorem}\label{substitutes}
For every  finite and closed set of surfaces $\Sbbb,$  where $\eg(\Sbbb)=g$, there exist functions $f_{\ref{substitutes}}^1,f_{\ref{substitutes}}^2:\mathbb{N}^2\to\mathbb{N}$ and an algorithm  that, given a graph $G$ and a $k\in\mathbb{N}$, outputs either a  Dyck-grid in $\mathfrak{D}_{\mathsf{sobs}(\Sbbb)}\coloneqq \{ \mathscr{D}^{\Sigma}\mid \Sigma\in\mathsf{sobs}(\Sbbb)\}$ 
of order $k$ as a minor of $G$ or a tree decomposition of $G$ where each torso has an $f_{\ref{substitutes}}^1(g,k)$-bidimensional modulator to $\Ecal_{\Sbbb}$.
Moreover, the algorithm runs in time $\Ocal(f_{\ref{substitutes}}^2(g,k)\cdot |V(G)|^{4}\log (|V(G)|)),$ where $f_{\ref{substitutes}}^1(g,k)= 2^{2^{\mathcal{O}(g)}\mathsf{poly}(k)}$ and $f_{\ref{substitutes}}^2(g,k)= 2^{2^{2^{\mathcal{O}(g)}\mathsf{poly}(k)}}$. 
\end{theorem}
In the above, we use the term $\mathsf{poly}(k)$ as a shortcut for $k^{\Ocal(1)}$.

\subsection{Overview of the proof and the structure of the paper}

Towards establishing \cref{annihilation} 
we  prove, (in reverse order) that  $${\mathsf{sobs}(\Sbbb)\mbox{-}\bg}\preceq \Sbbb\mbox{-}\mathsf{\tw}\preceq \Sbbb\mbox{-}\mathsf{\hw}\preceq {\mathsf{sobs}(\Sbbb)\mbox{-}\bg}.$$

The  relation 
$\Sbbb\mbox{-}\mathsf{\hw}\preceq {\mathsf{sobs}(\Sbbb)\mbox{-}\bg}$ 
is essentially a reinterpretation of the 
 results of \cite{thilikos2026excluding} and is presented in  \cref{all_minors}. 
We refer to the fact that  $\Sbbb\mbox{-}\mathsf{\tw}\preceq \Sbbb\mbox{-}\mathsf{\text{$\hw$}}$ as ``the upper bound''
and this follows
from \cref{abstraction}, which is the main outcome of \cref{extrahuman}.
We finally refer to the relation  ${\mathsf{sobs}(\Sbbb)\mbox{-}\bg}\preceq \Sbbb\mbox{-}\mathsf{\tw}$, see \cref{predominant}, as ``the lower bound''  and we dedicate \cref{progressively} to its proof.
The results of   
\cref{all_minors}, \cref{extrahuman}, and \cref{predominant}
are assembled together in \cref{beneficial} in order to give a proof of \cref{annihilation} and \cref{main_mainl}.

Towards improving the readability of the paper we depict in \cref{tab_graph_parameters} all the parameters that are defined in this paper.

\begin{table}[ht]
\scalebox{.95}{{
\begin{minipage}{17cm} 
{\small 
\renewcommand{\arraystretch}{1}
\begin{tabular}{p{1.6cm}p{7.5cm}p{6.8cm}}
\hline
\textbf{Parameter} & \textbf{Definition} & \textbf{Interpretation} \\
\hline
$\bdim(G,X)$ &
Maximum $k$ such that $G$ contains the $(k\times k)$-grid as an $X$-rooted minor. &
Bidimensionality of the annotated set $X$ in $G$. \\

$\hw(G)$ &
Maximum $t$ such that $G$ contains $K_t$ as a minor. &
Hadwiger number. \\

$\hw'(G)$ &
Minimum $k$ such that $G$ belongs to the clique-sum closure of graphs with a $k$-bidimensional modulator to graphs of Euler-genus at most $k$. &
Decomposition-based counterpart of $\hw$. \\

%$\tw(G)$ &
%Minimum width of a tree decomposition of $G$. &
%Treewidth. \\

$g\mbox{-}\tw(G)$ &
Minimum $k$ such that $G$ belongs to the clique-sum closure of graphs with a $k$-bidimensional modulator to $\mathcal E_g$. &
Treewidth relative to surfaces of Euler-genus  $\leq g$. \\

$g\mbox{-}\bg(G)$ &
Maximum order $k$ of the relevant Dyck-grid minors prescribed by the Euler-genus case distinction. &
Biggest-grid analogue for Euler-genus $g$. \\

$\mathbb S\mbox{-}\tw(G)$ &
Minimum $k$ such that $G$ belongs to the clique-sum closure of graphs with a $k$-bidimensional modulator to $\mathcal E_{\mathbb S}$. &
Treewidth relative to a closed set of surfaces $\mathbb S$. \\

$\mathbb S\mbox{-}\bg(G)$ &
$\max\{k \mid G \text{ contains } \mathscr D_k^\Sigma \text{ as a minor for some } \Sigma\in\mathbb S\}$. &
Biggest Dyck-grid minor from a prescribed set of surfaces. \\
%
%$\mathsf p_{\mathfrak H}(G)$ &
%Maximum $k$ such that $G$ contains $\mathscr H_k^{(j)}$ as a minor for some $\mathscr H^{(j)}\in\mathfrak H$. &
%Parameter induced by a minor-parametric obstruction family. \\

%$\mathsf p(G,X)$ &
%$\max\{\mathsf p(H)\mid H \text{ is an } X\text{-minor of }G\}$. &
%Annotated extension of a minor-monotone parameter $\mathsf p$. \\
%
%$\mathsf p^{\mathsf t}(G,X)$ &
%$\min\{\mathsf p(\mathsf{torso}(G,Z))\mid X\subseteq Z\subseteq V(G)\}$. &
%Torso-extension of $\mathsf p$. \\
%
%$\mathsf p^\star(G)$ &
%Minimum $k$ such that $G$ belongs to the clique-sum closure of graphs with $\mathsf p$-value at most $k$. &
%Clique-sum extension of $\mathsf p$. \\
%
%$\mathsf p_{\mathcal G}(G)$ &
%$\min\{k\mid \exists X\subseteq V(G): \mathsf p(G,X)\le k \text{ and } G-X\in\mathcal G\}$. &
%Modulator composition with the target class $\mathcal G$. \\
%
%$\mathsf{size}(G)$ &
%$|V(G)|$. &
%Number of vertices; used as a reference parameter. \\
\hline
\end{tabular}
}
\end{minipage}
}}
\caption{Main graph parameters.}
\label{tab_graph_parameters}
\end{table}

\section{Excluding a surface}
\label{all_minors}

In this subsection we introduce the results from \cite{thilikos2026excluding} that provide a refinement of the GMST where, instead of excluding a clique, we exclude 
a Dyck-grid corresponding to some fixed surface.

In \cref{inphasized}, give a short introduction to graph parameters and the concept of universal obstructions. 
In order to present the main result of this section, we first introduce 
the concept of  $\Sigma$-decompositions in \cref{s_decompo}.
For this, we use the terminology introduced by Kawarabayashi, Thomas, and Wollan in their proof of the GMST in \cite{KawarabayashiTW20Quicklyexcluding}.
We present the main outcome of \cite{thilikos2026excluding} in \cref{exc_grid} as well as a suitable translation of it according to the terminology that we introduce  in this paper.

\subsection{Graph parameters and Dyck-grids}\label{inphasized}

We denote by $\mathbb{Z}$ the set of integers and by $\mathbb{R}$ the set of reals.
Given two integers $a,b\in\mathbb{Z}$ we denote the set $\{z\in\mathbb{Z} \mid a\leq  z\leq  b\}$ by $[a,b].$
In case $a>b$ the set $[a,b]$ is empty. For an integer $p\geq  1,$ we set $[p]=[1,p]$ and $\mathbb{N}_{\geq  p}=\mathbb{N}\setminus [0,p-1].$
If $\Scal$ is a collection  of objects where the union operation is defined, then  
we define $\cupall\Scal=\bigcup_{S\in \Scal}S$.

\smallskip

All graphs considered in this paper are undirected, finite, and without loops or multiple edges.
We use standard graph-theoretic notation, and we refer the reader to~\cite{diestel2016graph} for any undefined terminology. 
Recall that we use $\mathcal{G}_{\text{all}}$ for the set of all graphs.

We continue by revisiting the notions and ideas touched upon in the introduction in a more formal way.

\paragraph{Graph parameters.}
  A \emph{graph parameter} is a function mapping $\mathcal{G}_{\text{all}}$ to non-negative integers.
 Let $\p$ and $\p'$ be two graph parameters. 
We write  $\p\preceq \p'$ in order to denote that there exists some function $f \colon \mathbb{N} \to \mathbb{N}$ such that, for every graph $G,$ it holds that $\p(G) \leq  f(\p'(G))$

 We say that $\p$ and $\p'$ are \emph{equivalent}, which we denote by $\p \sim \p'$, if   $\p\preceq \p'$ and $\p'\preceq \p$. If the functions involved in this definition are polynomial (resp. linear)  we say that $\p_{1}$ and $\p_{2}$ are \emph{polynomially equivalent} (resp. \emph{linearly equivalent}) and we denote this by $\p \sim_{\mathsf{P}} \p'$ (resp. $\p \sim_{\mathsf{L}} \p'$).
 \smallskip

Next, we introduce tree decompositions and the concept of clique-sum closures, which can be seen as another way to say that a graph has a tree decomposition where the torso of every bag satisfies some property.

\paragraph{tree decompositions}
  Let $G$ be a graph. 
  A \emph{tree decomposition} of $G$ is a tuple $\mathcal{T}=(T,\beta)$ where $T$ is a tree and $\beta\colon V(T)\to 2^{V(G)}$ is a function, whose images are called the \emph{bags} of $\mathcal{T},$ such that 
  \medskip
  \begin{enumerate}
  \item $\bigcup_{t\in V(T)} \beta (t)= V (G),$ 
  \item for every $e\in E(G)$ there exists $t\in V(T)$ with $e\subseteq \beta(t),$ and 
  \item for every $v\in V(G)$ the set $\{t\in V(T) \mid v\in \beta(t)\}$ induces a subtree of $T.$
  \end{enumerate} 
  \medskip
  We refer to the vertices of $T$ as the \emph{nodes} of the tree decomposition $\mathcal{T}.$

  For each $t\in V(T),$ we define the \emph{adhesions of $t$} as the sets in $$\{ \beta(t)\cap\beta(d) \mid d\text{ is adjacent with }t \}$$ and the maximum size of them is called the  \emph{adhesion of $t$}.
  The \emph{adhesion} of $\mathcal{T}$ is the maximum adhesion of a node of $\mathcal{T}.$
  The \emph{torso} of $\mathcal{T}$ on a node $t$ is the graph, denoted by $G_{t},$ obtained by adding edges between every pair of vertices of $\beta(t)$ which belongs to a common adhesion of $t.$
   The \emph{clique-sum closure} of a graph class $\mathcal{G}$ is the graph class containing every graph that has a tree decomposition whose torsos belong to $\mathcal{G}.$

  The \emph{width} of a $(T,\beta)$ is the value $\max_{t\in V(T)}|\beta(t)|-1.$
  The \emph{treewidth} of $G,$ denoted by $\tw(G),$ is the minimum width over all tree decompositions of $G.$ In other words a graph $G$ has treewidth at most $k$ iff $G$ belongs in the clique-sum closure of the class of graphs containing at most $k+1$ vertices.
 
\paragraph{Parametric graphs and universal obstructions.}
Next, we formalize the concept of parametric graphs that can act as asymptotic obstructions to graph parameters.

A \emph{parametric graph} is a sequence of graphs $\mathscr{H}=\langle \mathscr{H}_{i}\rangle_{i\in\mathbb{N}}$ indexed by non-negative integers.
We say that $\mathscr{H}$ is \emph{minor-monotone} if for every $i\in\mathbb{N}$ we have $\mathscr{H}_i \leq  \mathscr{H}_{i+1}.$
All parametric graphs considered in this paper are minor-monotone.

Let $\mathscr{H}^1=\langle \mathscr{H}_{k}^1\rangle_{k\in \mathbb{N}}$ and $\mathscr{H}^2=\langle \mathscr{H}_{k}^2\rangle_{k\in \mathbb{N}}$ be two parametric graphs.
We write $\mathscr{H}^1\lesssim \mathscr{H}^2$ if there is a function $f\colon\mathbb{N}\to\mathbb{N}$ such that, for every $k\in\mathbb{N},$ $\mathscr{H}_{k}^1 \leq   \mathscr{H}_{f(k)}^2.$
If $f$ is a polynomial (resp. linear) function, then we write  
$\mathscr{H}^1\lesssim_{\mathsf{P}} \mathscr{H}^2$ (resp. $\mathscr{H}^1\lesssim_\mathsf{L} \mathscr{H}^2$).
We say that $\mathscr{H}^1$ and $\mathscr{H}^2$ are 
\emph{equivalent} if $\mathscr{H}^1 \lesssim \mathscr{H}^2$
and $\mathscr{H}^2 \lesssim \mathscr{H}^1$, and we denote this by $\mathscr{H}^1 \approx \mathscr{H}^2.$ The notion of \emph{polynomial (resp. linear)
 equivalence} is denoted $\mathscr{H}^1 \approx_{\mathsf{P}} \mathscr{H}^2$ (resp. $\mathscr{H}^1 \approx_{\mathsf{L}} \mathscr{H}^2$) and is  defined in the obvious way.

For instance, the parametric 
graphs of the grids or of the walls (defined in \cref{seb_calls_def}) are both linearly equivalent to annulus grids.

A \emph{minor}-\emph{parametric family} is a finite collection $\mathfrak{H}=\{ \mathscr{H}^{(j)} \mid j\in[r] \}$ of (minor-monotone) parametric graphs with $\mathscr{H}^{(j)}=\langle \mathscr{H}^{(j)}_{i}\rangle_{i\in\mathbb{N}}$ for all $j\in[k],$ and for all distinct $i,j\in[r],$ $\mathscr{H}^{(i)}\not\lesssim\mathscr{H}^{(j)}$ and $\mathscr{H}^{(j)}\not\lesssim\mathscr{H}^{(i)}.$
All parametric families considered in this paper are minor-monotone.
The notions $\lesssim_{\mathsf{L}},$  $\lesssim_{\mathsf{P}},$ $\equiv_{\mathsf{L}},$  and $\equiv_{\mathsf{P}}$ are defined analogously. 
\medskip

For every minor-parametric family $\mathfrak{H}$, we define the parameter $\p_{\mathfrak{H}}$ such that $\p_{\mathfrak{H}}(G)$ is the maximum $k$ for which, there exists some $j\in[r],$ such that $\mathscr{H}^{j}_{k}\leq G.$

We say that  a {{minor}-{parametric family}} $\mathfrak{H}=\{\mathscr{H}^{(j)} \mid j\in[r] \}$ is a \emph{minor-universal obstruction} (or \emph{universal obstruction} if it is clear from the context that we are working with the minor relation) of some graph parameter $\mathsf{p}$ if $\p\sim\p_{\mathfrak{H}}.$

Recall that in \cref{undergoing} we defined the {minor}-{parametric family} $\mathfrak{D}_{\mathsf{sobs}(\Sbbb)} = \{\mathscr{D}^{\Sigma}\mid \Sigma\in\mathsf{sobs}(\Sbbb)\}.$ 
Given the above notation, \cref{unrelatedly} can be restated as follows.
\begin{corollary}
\label{chaotically}
For every  finite and   closed set of surfaces $\Sbbb,$ the set 
$\mathfrak{D}_{\mathsf{sobs}(\Sbbb)}$ is a minor-universal obstruction for ${\Sbbb}\mbox{-}\mathsf{\tw}.$
\end{corollary}

For more on the foundation of universal obstructions of parameters, see \cite{paul2023graph} (see also \cite{paul2023universal} for a recent survey).

\subsection{$\Sigma$-decompositions}
\label{s_decompo}

% We start by introducing the notions we will be using to describe ``almost embeddings'' in surfaces.

\paragraph{Societies.}
Let $\Omega$ be a cyclic permutation of the elements of some set which we denote by $V(\Omega)$.
A \emph{society} is a pair $(G,\Omega)$, where $G$ is a graph and $\Omega$ is a cyclic permutation with $V(\Omega)\subseteq V(G)$.
% A \emph{segment} of~$\Omega$ is a set~${S \subseteq V(\Omega)}$ such that there do not exist~${s_1,s_2 \in S}$ and~${t_1,t_2 \in V(\Omega) \setminus S}$ such that~${s_1,t_1,s_2,t_2}$ occur in~$\Omega$ in the order listed.  
% We define the \emph{depth} of~${(G,\Omega)}$ as the maximum number of pairwise vertex-disjoint paths
% from a segment $S$ of $\Omega$ to the vertices of $V(\Omega)\setminus S$.
A \emph{linear decomposition} of $(G,\Omega)$ is a sequence  $\mathcal{L}= \langle X_1,X_2,\dots,X_n,v_1,v_2,\dots,v_n \rangle$ such that $v_1,v_2,\dots,v_n\in V(\Omega)$ occur in that order on $\Omega$ and $X_1,X_2,\dots,X_n$ are subsets of $V(G)$ such that
  \begin{itemize}[itemsep=-2pt]
  \setlength\itemsep{0em}
    \item  for all $i\in[n]$, $v_i\in X_i$,
    \item $\bigcup_{i\in[n]}X_i=V(G)$, 
    \item for every $uv\in E(G)$, there exists $i\in [n]$ such that $u,v\in X_i$, and 
    \item for all $x\in V(G)$, the set $\{ i\in[n] \mid x\in X_i \}$ forms an interval in $[n]$.
  \end{itemize}
The \emph{width} of a linear decomposition is $\max\{|X_i|\mid {i\in [n]}\}$.

\paragraph{Drawings in a surface.}
  A \emph{drawing} (with crossings) \emph{in a surface $\Sigma$} is a triple $\Gamma=(U,V,E)$ such that
  \begin{itemize} 
  \item $V$ and $E$ are finite, and  $V\subseteq U\subseteq\Sigma$, 
  \item $V\cup\bigcup_{e\in E}e=U$ and $V\cap (\bigcup_{e\in E}e)=\emptyset$, 
  \item for every $e\in E$,  $e=h((0,1))$, where $h\colon[0,1]_{\mathbb{R}}\to U$ is a homeomorphism onto its image with $h(0),h(1)\in V$, and
  \item if $e,e'\in E$ are distinct, then $|e\cap e'|$ is finite.
  \end{itemize}
  We call the set $V$  the \emph{nodes} of $\Gamma$ and the set $E$ the \emph{arcs} of $\Gamma$.
  If $G$ is graph and $\Gamma=(U,V,E)$ is a drawing with crossings in a surface $\Sigma$ such that $V$ and $E$ naturally correspond to $V(G)$ and $E(G)$ respectively, we say that $\Gamma$ is a \emph{drawing of $G$ in $\Sigma$}. If no two edges in $E$ have a common point, we say that $\Gamma$ is an \emph{embedding} of $G$ in $\Sigma$. 
  In this last case, the connected components of $\Sigma\setminus U$, are the \emph{faces} of $\Gamma$.

\paragraph{$\Sigma$-decomposition.} 
  Let $\Sigma$ be a surface. 
  A \emph{$\Sigma$-decomposition} of a graph $G$ is a pair $\delta =(\Gamma,\mathcal{D})$, where $\Gamma=(U,V,E)$ is a drawing of $G$ in $\Sigma$ with crossings, and $\mathcal{D}$ is a collection of closed disks, each a subset of $\Sigma$ such that 
  \begin{enumerate}
  \item the disks in $\mathcal{D}$ have pairwise disjoint interiors, 
  \item the boundary of each disk in $\mathcal{D}$ intersects $\Gamma$ in nodes only, 
  \item if $\Delta_1,\Delta_2\in\mathcal{D}$ are distinct, then $\Delta_1\cap\Delta_2\subseteq V$, and 
  \item every arc of $\Gamma$ belongs to the interior of one of the disks in $\mathcal{D}$. 
  \end{enumerate}  
  If $\Delta\in\mathcal{D}$,  
%
%\vspace{3cm}
%
%
%
  and $c$ is the set created 
 from $\Delta$ by removing all nodes 
 drawn on its boundary,   
 then we say that $c$ is a \emph{cell} of $\delta$ and we use $\widetilde{c}$ to denote the set of  points where these nodes are drawn. A cell $c\in C(\delta )$ is called a \emph{vortex} if $|\widetilde{c}|\geq  4$.

 Given a cell $c$, we use $G_{c}$ for the subgraph of $G$ corresponding to the nodes and the arcs drawn in $c$ and we also use $\Omega_{c}$ for the cyclic ordering of the vertices drawn in the nodes of $\widetilde{c}$ respecting the order they appear.
 The width of $\delta$ is the maximum width of a society in $\{(G_{c},\Omega_{c})\mid \text{$c$ is a vortex  of $\delta$}\}$. The \emph{breadth} of $\delta$ is the number of its vortices.

\subsection{Excluding Dyck-grids}
\label{exc_grid}

Given two graphs $Z$ and $H$ 
we say that $Z$ is an \emph{$H$-expansion} if $Z$ contains $H$ as a minor and, moreover, any proper subgraph of $Z$ does not contain $H$ as a minor.
\medskip

We now present the following structural result that is an abbreviated version of \cite[Theorem 6.2]{thilikos2026excluding} also including its running time.

\begin{proposition}\label{thm_pre_globalstructure}
There exist functions $f_{\ref{thm_pre_globalstructure}}^1:\Nbbb^{5}\to\Nbbb$ and  $f_{\ref{thm_pre_globalstructure}}^1:\Nbbb^{5}\to\Nbbb$ such that for every $k\in\Nbbb$, $\mathsf{h}\in\Nbbb$, $\mathsf{h}'\in[\mathsf{h}-1]$, $\mathsf{c}\in[2]$, and for every graph $G$, 
given that 
 \begin{itemize}
 \item $g_1\coloneqq 2\mathsf{h}$ and 
 \item $g_2\coloneqq
 \begin{cases}  2\mathsf{h}'+\mathsf{c}
 &\mbox{if }  \mathsf{c}\neq 0 \\ 
 0 & \mbox{otherwise}
 \end{cases}$.
 \end{itemize} 
then one of the  following is true.
\begin{enumerate}
    \item $G$ contains a \textsl{(a)} $\mathscr{D}^{(\mathsf{h},0)}_k$ or a \textsl{(b)} $\mathscr{D}^{(\mathsf{h}',\mathsf{c})}_k$-expansion $D$ where case \textsl{(b)} is allowed as an outcome if and only if $g_2\neq 0$, or
    \item $G$ has a tree decomposition $(T,\beta)$ of adhesion less than $f_{\ref{thm_pre_globalstructure}}^{1}(\mathsf{h},k)$ such that  for every $d\in V(T)$,
     $|\beta(d)|\leq 4f_{\ref{thm_pre_globalstructure}}^1(\mathsf{h},k)$ or
   there exists a set $A_d\subseteq \beta(d)$   and     a surface $\Sigma$ where 
\begin{enumerate}      
      \item $|A_d|≤4f_{\ref{thm_pre_globalstructure}}^1(\mathsf{h},k)$
\item         neither 
    $\mathscr{D}^{(\mathsf{h},0)}_i$ 
    nor  (in case $\mathsf{c}'\neq 0$) $\mathscr{D}^{(\mathsf{h}',\mathsf{c})}_i$  embeds in $\Sigma$ for all $i\in\Nbbb$
    \item  for the torso $G_d$ of $G$ at $d$, the graph $G_d-A_d$ has a $\Sigma$-decomposition $\delta$ where 
	\begin{enumerate}
	    \item $\delta$ has width at most $f_{\ref{thm_pre_globalstructure}}^1(\mathsf{h},k)$ and breadth at most $f_{\ref{thm_pre_globalstructure}}^2(\mathsf{h},k)$,
    	    \item there is no vertex of $G_d-A_d$ which is drawn in the interior of a non-vortex cell of $\delta$,
   	     \item $G_d-A_d$ is a minor of $G$, and
   	     \item for every {neighbor} $d'$ of $d$ in $T,$ it holds that $|\big(\beta(d)\cap\beta(d')\big)\setminus A_d|\leq 3$.
	\end{enumerate}
\end{enumerate}
\end{enumerate}
Moreover, given that $g\coloneqq\max\{ g_1,g_2\}$, it holds that $f_{\ref{thm_pre_globalstructure}}^1(\mathsf{h},k)=2^{2^{\Ocal(g)}k^{\Ocal(1)}}$ and that 
$f_{\ref{thm_pre_globalstructure}}^2(\mathsf{h},k)=2^{\Ocal(g)}k^{\Ocal(1)}$.
Also, there is an algorithm, that, given $k$, $\mathsf{h}, \mathsf{h}', \mathsf{c}$, and $G$ as above, provides one of the above outcomes in $2^{2^{2^{\Ocal(g)}\cdot k^{\Ocal(1)}}}\cdot |V(G)|^4\cdot \log(|V(G)|)$ time.
\end{proposition}

Our next aim is to rewrite \cref{thm_pre_globalstructure} using the  terminology of surface obstruction introduced in this paper.

\begin{lemma}\label{thm_globalstructure}
There exists functions $f_{\ref{thm_globalstructure}}^1$ and $f_{\ref{thm_globalstructure}}^2$  
such that for every 
 $k\in\Nbbb$, for every closed set of surfaces $\Sbbb$, where $g:=\eg(\Sbbb)$, 
and every graph $G$, 
one of the following holds: 
\begin{enumerate}
    \item   $G$ contains as a subgraph an  $\mathscr{D}^{\Sigma'}_k$-{expansion},  for some $\Sigma'\in\sobs({\Sbbb})$ or 
    \item $G$ has a tree decomposition $(T,\beta)$ of adhesion at most  $f_{\ref{thm_globalstructure}}^1(g,k)$ such that, for every $d\in V(T)$, there exists a set $A_d\subseteq \beta(d)$, $|A_d| ≤f_{\ref{thm_globalstructure}}^1(g,k)$ and  a surface $\Sigma\in\Sbbb$ such that, given that $G_d$ is the torso of $G$ at $d$, the graph $G-A_d$ has a $\Sigma$-decomposition $\delta$ of width at most $f_{\ref{thm_globalstructure}}^1(g,k)$ and breadth at most $f_{\ref{thm_globalstructure}}^2(g,k)$ such that, if $X_d$ is the set of all vertices of $G_d-A_d$ drawn in the interior of the vortices of $\delta$ together with the vertices from $A_{d}$,
    \begin{enumerate}
        \item there is no vertex of $G_d-A_d$ which is drawn in the interior of a non-vortex cell of $\delta$,
        \item $G_d-X_d$ is a minor of $G$, and
        \item for every {neighbor} $d'$ of $d$ in $T$ it holds that $|\big(\beta(d)\cap\beta(d')\big)\setminus\big(A_d\cup X_d\big)|\leq 3$.
    \end{enumerate}
\end{enumerate}
Moreover it holds that $f_{\ref{thm_globalstructure}}^1(\mathsf{h},k)=2^{2^{\Ocal(g)}k^{\Ocal(1)})}$ and that 
$f_{\ref{thm_globalstructure}}^2(\mathsf{h},k)=2^{\Ocal(g)}k^{\Ocal(1)}$.
Also, there is an algorithm, that, given $k$, $\mathsf{h}, \mathsf{h}', \mathsf{c}$, and $G$ as above, provides one of the above outcomes in $2^{2^{2^{\Ocal(g)}\cdot k^{\Ocal(1)}}}\cdot |V(G)|^4\cdot \log(|V(G)|)$ time. 
\end{lemma}

\begin{proof}
To see how \cref{thm_globalstructure} is derived from \cref{thm_pre_globalstructure}, recall that $\sobs(\Sbbb)=\{\Sigma_{i}\mid i\in[q]\}$
contains $q\in[2]$ surfaces. If $q=1$, then $\Sigma_{1}$ is an orientable 
surface $\Sigma^{(\mathsf{h},0)}$. 
If $q=2$, then we may assume that $\Sigma_{1}$ is an orientable 
surface $\Sigma^{(\mathsf{h},0)}$ and $\Sigma_{2}$ is a non-orientable 
surface $\Sigma^{(\mathsf{h}',\mathsf{c})}$ for some $\mathsf{h}'\in[\mathsf{h}-1]$
and some $\mathsf{c}\in[2]$. As every surface in $\Sbbb$ is contained in some of the surfaces in the surface obstruction $\{\Sigma_{i}\mid i\in[q]\}$, we have that  
they all have Euler genus  smaller than $\eg(\Sbbb)$, therefore 
\cref{thm_pre_globalstructure} can be applied either for $\mathsf{h}$ and $k$ or for $\mathsf{h}$, $\mathsf{h}'$, $\mathsf{c}$, and $k$.
In its second outcome, \cref{thm_pre_globalstructure} gives a surface 
$\Sigma$ where not all graphs in 
$\mathscr{D}^{(\mathsf{h},0)}_i$ or (in case $\mathsf{c}\neq 0$) in 
$\mathscr{D}^{(\mathsf{h}',\mathsf{c})}_i$ embed. This, because of \cref{territorial}, implies that $\Sigma$ does not contain any of the surfaces in $\sobs(\Sbbb)$, therefore $\Sigma\in\Sbbb$. Notice also that 
$A_{d}\subseteq X_{d},$ therefore Conditions (a), (b), (c), follow 
because of conditions ii-iv of \cref{thm_pre_globalstructure}.
\end{proof}

\section{Bidimensionality and $\Sigma$-decompositions}
\label{extrahuman}

In this section we prove that, in a $\Sigma$-decomposition, 
the union of the apices and the vertices that are drawn inside the vortices 
has low bidimensionality.

\subsection{Walls in surfaces and their radii}
\label{seb_calls_def}

Our first aim is to provide some certificate for the big bidimensionality. This is done using the concept of a \textsl{rooted wall}, defined below.

An \emph{$(n \times  m)$-grid} is the Cartesian product of a path on $n$ vertices and a path on $m$ vertices. We denote the vertices of a $(n \times  m)$-grid as pairs $(i,j)$ where $i\in[n]$ and $j\in[m]$

The \emph{elementary $r$-wall} is the graph obtained from the $(r \times 2r)$-grid by deleting all edges of the form
\begin{itemize}
 \item $\{(i,j),(i+1,j) \}$ for $i\in[r-1]$ odd and $j\in[2r]$ even, and
 \item $\{(i,j),(i+1,j) \}$ for $i\in[r-1]$ even and $j\in[2r]$ odd.
\end{itemize}
An \emph{$r$-wall} is any graph that can be obtained from the elementary $r$-wall by subdividing its edges an arbitrary number of times. 
Notice that when $r≥3$ every $r$-wall has a unique, up to homeomorphism, embedding in the sphere.
Also, this embedding has a unique face whose boundary contains more than 6 non-subdivision vertices. We call the cycle corresponding to the boundary of this face \emph{perimeter} of the $r$-wall.
We call a path of a wall \emph{maximally induced} if all its internal vertices 
have degree 2 and its endpoints have degree $> 2$.

%We refer to the non-subdivision vertices of an $r$-wall as its \emph{original vertices}.

%
%
%
%
%
%
A \emph{rooted $r$-wall} is a pair $(W,X)$    defined as follows:
consider an $r$-wall $\overline{W}$ and a set $S\subseteq V(\overline{W}$)
such that every maximally induced path \\ 
 contains exactly one internal vertex from  $S$. Notice that this implies that 
$S$ contains only  subdivision vertices of $\overline{W}$. Next, 
for every $s\in S$ add a new vertex  $x_{s}$  and
add the edge $sx_s$. Finally, for each
such edge $sx_s$, either contract it or subdivide it some arbitrary number of times (possible none). The graph obtained after this path enhancement is the graph $W$ and
the set $X$ consists of every vertex $x_{s}$. We also agree that, in the case of a contraction, $x_{s}=s$. 
We call the newly introduced paths with endpoints $s$ and $x_{s}$ \emph{hairs} of the
resulting pair $({W},X)$. Clearly, 
in case $s=x_{s}$, this path is a trivial path.

The following observation follows easily by the fact that every $r$-wall is a subdivision of a subgraph of a $(2k\times 2k)$-grid.

\begin{observation}\label{uncoordinated}
There is a universal constant $c_{\ref{uncoordinated}}$ such that for every graph $G$ and every set $X\subseteq V(G),$ if $\bdim(G,X)\geq  c_{\ref{uncoordinated}}{k},$ then $G$ contains a rooted $k$-wall $(W,X)$  as a subgraph.
\end{observation}

\cref{uncoordinated} permits us to use rooted $r$-walls as certificates of big bidimensionality. Observe also that every rooted $r$-wall $({W},X)$ contains an  $(r\times r)$-grid as an $X$-minor. Moreover,  it is easy to verify that if  a graph $G$ contains an  $(4r\times 4r)$-grid 
as an $X$-minor for some $X\subseteq V(G)$, then $G$ contains a subgraph $W$ and a set $X'\subseteq X$
such that $(W,X')$ is an $r$-wall of $(G,X)$. Therefore, 
the presence (resp. absence) of big rooted walls can be seen as an alternative, and sometimes more convenient, way to see the big (resp. small) bidimensionality of a set.

\begin{figure}[ht]
\begin{center}
\includegraphics[height = 4cm]{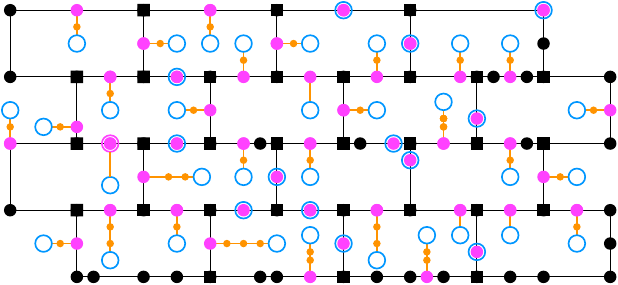}
\end{center}
  \caption{A {rooted $5$-wall} $(\widehat{W},X).$ The edges and the internal nodes of its hairs are \darkorange{orange}. The set $S$ is depicted by the \textcolor{DeepMagenta}{magenta} vertices and the set $X$ is depicted by the \textcolor{CornflowerBlue}{blue} circles. The branch vertices are the square vertices.}
\label{experienced}
\end{figure}

We say that a vertex of a graph is a \emph{branch vertex} if it has degree at least 3.
Given two vertices $x,y,$ we define their \emph{branch distance}, denoted by $\mathsf{bdist}_{G}(x,y)$ as the minimum number of branch vertices of a path between $x$ and $y.$
The  \emph{ branch radius} of $G$ is at most $r$ if there is an $x\in V(G)$ such that for every $y\in V(G),$ $\mathsf{bdist}_{G}(x,y)\leq  r.$

\begin{observation}
\label{obs_critic}
Let $(\widehat{W},X)$ be a rooted $r$-wall and let $Q$ be the graph obtained by adding in $\widehat{W}$ a new vertex $a$ and making it adjacent to all vertices of $X.$
Then $Q$ has branch radius at most two.
\end{observation}

Let $\Sigma$ be a surface without boundary and let $G$ be a graph embedded in $\Sigma$. Towards simplifying notation, we do not distinguish the graph and its particular embedding, and we use $F(G)$ to denote the set of faces of $G,$ i.e., the set of connected components of the set $\Sigma\setminus G.$
Given $x,y\in V(G)\cup F(G),$ a \emph{radial $(x,y)$-path} in $G$ of length at most $r$ is a sequence of vertices and faces of length at most $r+1,$ where faces and vertices alternate, which starts from $x$ and finishes in $y,$ and such that if $f\in F(G)$ and $v\in V(G)$ are consecutive in the sequence, then $v$ is a point of the boundary of $f.$
The \emph{radial distance} between $x$ and $y,$ denoted by $\mathsf{rdist}_{G,\Sigma}(x,y)$ is the minimum length of a radial $(x,y)$-path in the embedding of $G$ in $\Sigma.$
The \emph{radial radius} of $G$ is at most $r$ if there is an $x\in V(G)\cup F(G)$ such that for every $y\in V(G)\cup F(G),$ $\mathsf{rdist}_{G,\Sigma}(x,y)\leq  r.$

The following has been proven in different forms, see \cite[Lemma 4.3]{DeHaTh06}, \cite[Lemma 4]{FominGT11contr}, or \cite[Lemma 4]{GeelenRS04embed}.

\begin{proposition}\label{dilettantism}
  Let $G$ be a graph embedded in a surface $\Sigma$ of Euler-genus $g.$
  There is a universal constant $c$ such that if the treewidth of G is more than $cr(g + 1),$ then $G$ contains an $r$-wall $W$ as a subgraph whose embedding is a subset of a  closed disk of $\Sigma$ whose boundary is the perimeter of $W.$ 
\end{proposition}

\begin{lemma}\label{compensation}
  Let $G$ be a graph embedded in a surface $\Sigma$ of Euler-genus at most $g$ with radial radius at most $r.$
  Then $\tw(G)\in\mathcal{O}(g\cdot r)$
\end{lemma}
  
\begin{proof}
  Let $q$ be the maximum integer such that $G$ contains a $2(q+1)$-wall $W$ as a subgraph that is a subset of a closed disk of $\Sigma$ whose boundary is the perimeter of $W.$
  Notice that the radial radius of $W$ is more than $q,$ which in turn implies that also the radial radius of $G$ is at least $q.$
  As $G$ has radial radius at most $r,$ we have that $q\leq  r.$
  This, combined with \cref{dilettantism} implies that $\tw(G)\in\mathcal{O}(g\cdot r).$
\end{proof}

\subsection{Bounding the bidimensionality}

Let us denote the $(k\times k)$-grid by $\Gamma_{k}.$
The \emph{perimeter} of $\Gamma_{k}, k\geq  3,$ is defined in the same way as it is defined for walls.
We say that the \emph{$1$st layer} of $\Gamma_{k}$ is its perimeter.
Given some $s\leq  \lfloor k/2\rfloor$ we define the \emph{$s$-th layer} of $\Gamma_{k}$ as the perimeter of the $\big((k-2s+2)\times (k-2s+2)\big)$-grid remaining after removing the the first $s-1$ layers of $\Gamma_{k}.$
We refer to the first $s$ layers of $\Gamma_{k}$ as its $s$ \emph{outermost layers}.

Given an $(m\times m)$-grid $\Gamma_{m}, m\geq 3,$ and some $k$ where $3\leq  k \leq  m$ and such that $m-k$ is even, we  define the \emph{central $(k\times k)$-subgrid} of $\Gamma_{m}$ as the unique subgraph of $\Gamma_{m}$ obtained by removing its ${\frac{m-k}{2}}$ outermost layers.

Let $G$ and $H$ be graphs, $X\subseteq V(G)$, and assume that $H$ is an {$X$-minor} of $G,$ certified by  the collection $\mathcal{S}=\{S_{v}\mid v\in V(H)\}$ of pairwise-disjoint {connected} subsets of $V(G).$
If $x\in S_{v},$ then we say that $x$ is a \emph{precursor}
of $v$ in $G$ and that $v$ is the \emph{heir} of $x$ in $H.$ 

\begin{observation}\label{imperiously}
   Let $(H,Y)$ and $(G,X)$ be annotated graphs where $H$ is an $X$-minor of $G$. Then $\bdim(G,X)\geq  \bdim(H,Y).$
\end{observation}
    
An important tool for the study of bidimensionality was proved in \cite{DemaineFHT04bidim} and  is the following.

\begin{proposition} \label{subjection}
  Let $\Gamma_{m}$ be the $(m\times m)$-grid and let $S$ be a subset of the vertices in the central $(m-2\ell) \times (m-2\ell)$-subgrid of $\Gamma_{m},$ where $\ell=\lfloor \sqrt[4]{|S|} \rfloor.$ Then $G$ contains the $(\ell\times\ell)$-grid as an $S$-minor.
\end{proposition}

\cref{subjection} will be used for the proofs of most of the results in the rest of this paper. A sample of its use is given by the following lemma.

\begin{lemma}\label{acknowledged}
Let $G$ be a graph and let $X_{1},\ldots,X_{r}\subseteq V(G)$ such that, for $i \in[r],$ $\bdim(G,X_{i})\leq  \ell.$
Then $\bdim(G,X_{1}\cup\cdots\cup X_{r})\in\mathcal{O}(r^{4}\ell^2).$
\end{lemma}   
   
\begin{proof}
Let $X=X_{1}\cup\cdots\cup X_{r}$ and let $q=\bdim(G,X).$
Let also $\Gamma_{q}$ be a $(q\times q)$-grid $X$-minor in $G,$ certified by a collection $\mathcal{S}=\{S_{v}\mid v\in V(\Gamma_{q})\}$ of pairwise-disjoint {connected} subsets of $V(G).$ We choose  $i$ so that  the cardinality of the set $Z_i\coloneqq \{u\in V(\Gamma_q) \mid X_{i}\cap S_{u}\neq \emptyset\}$ is maximized, i.e., $X_{i}$ has the majority of intersections with the sets of $\mathcal{S}.$
As we fix $i,$ for simplicity we use $Z=Z_{i}.$

Clearly $Z$ is a subset of $V(\Gamma_{q})$ where $|Z|\geq  q^2/r.$ For each (positive) margin $b,$ we define the set $B_{b}$ as the union of the $b$ {outermost} layers of $\Gamma_q.$
We now fix some $b\in[\lfloor q/2\rfloor],$ and observe that at most $4b(q-b)$ of the vertices of $\Gamma_{q}$ are in the {outermost} $b$ layers. 
This means that at least $f(q,b,r)\coloneqq \frac{q^2}{r}-4b(q-b)$ vertices of $Z$ are in the central $(q-2b) \times (q-2b)$-subgrid of $\Gamma_{m}.$

We choose the margin $b_{q,r}$ as the maximum integer $b$ such that $\lfloor \sqrt[4]{f(q,b,r)}\rfloor\geq  2b.$

This implies that there are at least $f(q,b_{q,r},r)$ vertices from $Z$ in the $(q-2b_{q,r}) \times (q-2b_{q,r})$-subgrid of $\Gamma_{q}.$
By \cref{subjection} we obtain, $\bdim(G,X_{i})\geq \bdim(\Gamma_{q},Z_{i})\geq  \lfloor \sqrt[4]{f(q,b_{q,r},r)}\rfloor\geq  2b_{q,r}\in\Omega(\frac{q^{1/2}}{{r}^{1/4}}).$
Therefore, $\bdim(G,X_{1}\cup\cdots\cup X_{r})\in\mathcal{O}(r^4\ell^2).$
\end{proof}

Let $\delta$ be a $\Sigma$-decomposition of a graph $G$ and let $c$ be one of its cells. 
We define
\begin{eqnarray}
\mathsf{neigh}_{\delta}(c)  =  V(G_{c})\cup\{v\mid \mbox{$v$ is  drawn in the boundary of a non-vortex  cell $c'$ of $\delta$ where $\tilde{c}'\cap \tilde{c}\neq\emptyset$}\}.
\label{neigh_plus}\end{eqnarray}

In other words $\mathsf{neigh}_{\delta}(c)$ contains all vertices drawn 
inside $c$ as well as all vertices drawn on the boundaries of the neighboring cells of $c$.

The next observation follows directly from the definition of $\bdim.$

\begin{observation}\label{unrestricted}
  Let $G$ be a graph, $A\subseteq V(G),$ and $X\subseteq V(G)\setminus A,$ where $\bdim(G-A,X)\leq  k.$
  Then $\bdim(G,X\cup A)\leq  k+|A|.$
\end{observation}

\begin{lemma}\label{abstraction}
  For every graph $G,$ every $k\in \mathbb{N},$ and every surface $\Sigma$ with Euler-genus at most $g,$ if 
\begin{itemize}  
\item   $A\subseteq V(G),$ 
\item $\delta $ is a $\Sigma$-decomposition of $G-A$ of width $w$ and breadth $b$, 
\item ${\displaystyle X\coloneqq A\cup \cupall \{\mathsf{neigh}_{\delta}(c)}\mid \mbox{$c$ is a vortex of $\delta$}\}$
\end{itemize}
  then $\bdim(G,X)\in\mathcal{O}(b^4\cdot (b\cdot g\cdot w)^2+|A|).$
\end{lemma}

\begin{proof}
Let  $c_1,\ldots,c_{b}$ be the vortices of $\delta .$ 
Because of \cref{unrestricted},  we may assume that $A=\emptyset$ and prove that $\bdim(G,X)\in\mathcal{O}(b^{4}\cdot (b\cdot g\cdot w)^2).$
Because of \cref{acknowledged}, it is enough to prove that if $\widehat{c}\in \{c_1,\ldots,c_{b}\}$ and given that $S:=\mathsf{neigh}_{\delta}(\widehat{c}),$ then it holds that $\bdim(G,S)\in\mathcal{O}(bgw).$
Let $z=\lfloor\bdim(G,S)/{c'} \rfloor$ where $c'$ is the constant of \cref{uncoordinated}.

According to \cref{uncoordinated} $G$ contains a rooted $z$-wall $({W},S).$ 
Moreover, $\tw({W})\geq  z.$

As a first step, we enhance $G$ by adding in it, for each vortex $c_i\in \{c_1,\ldots,c_{b}\},$ a cycle ${C}_{i}$ joining every pair of consecutive vertices in $\Omega_{c_i}.$
We further enhance $G$ by adding in it  
a new vertex $a$ and making it adjacent to all vertices of $S$ that are drawn inside $\widehat{c}$.

We also update $\delta$ to a $\Sigma$-decomposition of the enhanced $G$  so that 
each cycle $C_{i}$ is  drawn inside $c_{i}$ and moreover the new vertex $a$ and its incident edges are all drawn inside $\widehat{c}$.

Notice that the above enhancement of $G$ does not increase the width of the vortices of $\delta $ by more than two units.
We now define the pair $(W^+,S)$ where $$W^+\coloneqq{W}\cup\cupall\{C_{i}\mid  i\in[b]\}.$$
Clearly, for every $i\in [b]$, $W^+$ may have vertices outside $G_{c_i}$.  Keep also in mind that all vertices and edges in $V(C_{i})$ are vertices and edges of both $W^+$ and $G_{c_{i}}.$
We now consider the graphs 
\begin{itemize} 
\item $H_{1}=W^{+}\setminus\cupall\{V(G_{c_i})\setminus V(C_{i})\mid  i\in[b]\},$ and 
\item $H_2=H_1\cup\cupall\{G_{c_i}\mid  i\in[b]\}.$
\end{itemize}

Observe that $H_1$ is a subgraph of $W^+$ and $W^+$ is a subgraph of $H_2.$
Also keep in mind that  $\tw(W^+)\geq  \tw(W)\geq  z.$

By construction, $H_1$ is a graph that is embedded in $\Sigma$.
Pick $j\in[b]$ such that $c_{j}=\widehat{c}$ and we denote by $C_{\widehat{c}}$ the cycle $C_{j}$.
Let $f^{\star}$ be the face of this embedding  whose boundary corresponds 
to the vertices and the edges of  $C_{\widehat{c}}.$ 
By the definition of $S\coloneqq \mathsf{neigh}_{\delta}(\widehat{c})$
all vertices of $S\cap V(H_{1})$ are drawn in the boundaries of cells sharing 
some point with $\widehat{c}$. This fact, 
by applying \cref{obs_critic} for the vertex $a$, implies  that, in the embedding of $H_{1}$ in $\Sigma$, all vertices of $S$  are  within radial distance $\Ocal(1)$  from $f^{\star}$. 
Notice also that  every non-branch vertex of $W$ that became a branch vertex of $W^+$
is a vertex of some cycle in $\{C_{1},\ldots,C_{b}\}$, therefore every such vertex
should also be  within radial distance $\mathcal{O}(b)$ from $f^{\star}$.
We just proved that all vertices of $H_{1}$ are within radial distance  $O(1)$ from $f$. This implies that the radius of $H_{1}$ is $\mathcal{O}(b)$ and this, combined with \cref{compensation} implies that  $\tw(H_1)\in\mathcal{O}(bg).$

Our next step is to use the fact that $\tw(H_1)\in \mathcal{O}(bg)$ in order to prove that $\tw(H_2)\in \mathcal{O}(bgw).$
For this, let $(T,\beta)$ be a tree decomposition of $H_1$ and let $L^i=\langle X_1^i,X_2^i,\dots,X_{n_i}^i,x_1^i,x_2^i,\dots,x_{n_i}^i\rangle$ be a linear decomposition of 
$\sigma(c_i),$ assuming that $x_1^i,x_2^i,\dots,x_{n_i}^i$ are the vertices of $V(\Omega_{c_i})$ ordered according to $\Omega_{c_i}.$
We now build a tree decomposition $(T,\beta')$ of $H_{2}$ where for each $t\in V(T),$ we define $\beta'(t)$ to be the union of the set $\beta(t)$ and every set $X_{j}^{i},$ for which $x_{j}^{i}$ is a vertex of $\beta(t)\cap \cupall\{\{x_1^i,x_2^i,\dots,x_{n_i}^i\}\mid i\in [b]\}.$
As each $X_{j}^{i}$ has $\mathcal{O}(w)$ vertices, it follows that $(T,\beta')$ has width $\mathcal{O}(bgw),$ as required.
We conclude that $\tw(H_2)\in \mathcal{O}(bgw).$ 

Finally, recall that $W^+$ is a subgraph of $H_2,$ therefore $z\leq  \tw(W^+)\leq  \tw(H_{2})\in \mathcal{O}(bgw).$
In other words, $\bdim(G,S)=c'z^2\in \mathcal{O}(bgw),$ as required.
\end{proof}

% \begin{eqnarray}
% \begin{minipage}{14cm}
% $\Sbbb\mbox{-}\hw(G)$ is the minimum $k$ for which $G$ is in the clique-sum closure of the graphs that contain a set of at most $k$ vertices whose removal yields a graph admitting a $\Sigma$-decomposition, of breadth at most $k$ and depth at most $k,$ where $\Sigma\in\Sbbb.$
% \end{minipage}\label{eradication}
% \end{eqnarray}

% \begin{eqnarray}
% {\Sbbb}\mbox{-}\tw(G)\coloneqq\min\{k\mid \mbox{$G$ belongs in the clique-sum closure of the graphs}\label{attempting} \\
% \mbox{with  a $k$-bidimensional modulator to $\Ecal_{\Sbbb}\}$}.\!\!\!\nonumber
% \end{eqnarray} 

%  $\Sbbb\mbox{-}\mathsf{\tw}\preceq \Sbbb\mbox{-}\mathsf{\text{$\hw$}}.$

%
%
%
%
%
%
Given a graph $G$, a set $A\subseteq V(G)$, where $|A|≤k$, and 
a $\Sigma$-decomposition
$\delta$ of  $G-A$ of width and breadth at most $k$, for some surface of Euler genus $≤k$, \cref{abstraction}
tells us that the union of $A$ with the set of  vertices drawn in the vortices of $\delta$ 
has bidimensionality $\Ocal(k^{10})$ in $G$. Actually,  using \cref{abstraction}, 
we can also add in this set the vertices drawn in the boundaries of all the non-vortex cells that are neighboring with the vortices (we will use this slightly stronger statement of  \cref{abstraction}
 in further investigations). According to the 
definitions of $\Sbbb\mbox{-}\hw(G)$ in \eqref{eradication} and $\Sbbb\mbox{-}\tw(G)$ in 
\eqref{attempting}, 
the immediate corollary 
of \cref{abstraction} is the following.

\begin{lemma}
\label{tiko_news}
There exists some polynomial function $f_{\ref{tiko_news}}:\Nbbb\to\Nbbb$ such that 
for every graph $G$, if $\Sbbb\mbox{-}\mathsf{\hw}(G)≤k$, then $\Sbbb\mbox{-}\mathsf{\tw}(G)≤f(k)$
\end{lemma}

\section{The lower bound}\label{progressively}

Recall that we defined $\Ecal_{\Sigma}$
as the class of graphs embeddable in $\Sigma$.

For the purposes of our lower bound,  
we introduce the parameter $\bdim_{\Sigma}(G):\mathcal{G}\to\Nbbb$ where
\begin{eqnarray}
\bdim_{\Sigma}(G) & = & \min\{k\mid \text{there exists } X\text{ s.t } \bdim(G,X)\leq k\mbox{~and~}G-X\in \Ecal_{\Sigma}\}\label{bidimsdefpre}
\end{eqnarray}
and the parameter 
$\bdim^{\star}_{\Sigma}(G):\mathcal{G}\to\Nbbb$,
where 
\begin{eqnarray}
    \bdim_{\Sigma}^{\star}(G) & = & \min\{k\mid \mbox{$G$ is the clique-sum closure of the class of }\label{bidimsdef}\\
   & &  \mbox{~~~~~~~graphs where the value of $\bdim_{\Sigma}$ is upper bounded by $k$}\}\label{new_bid_star}.
\end{eqnarray}

\begin{observation}\label{objectivity}
 $\bdim_{\Sigma}$ as well as $\bdim^{\star}_{\Sigma}$ are minor-monotone parameters, for every choice of $\Sigma$
\end{observation}

Our proof  has two steps. In Subsection \ref{overturning}
we prove that $\bdim_{\Sigma}(\mathscr{D}_{k}^{(\mathsf{h},\mathsf{c})})\in \Omega(k^{1/5}).$
Based on this result, in Subsection \ref{righteousness} we prove that 
$\bdim_{\Sigma}^{\star}(\mathscr{D}_{k}^{(\mathsf{h},\mathsf{c})})\in \Omega(k^{1/480}).$ In our proofs we do not make any particular effort to optimize the parametric dependences. While we keep them 
being polynomial, it seems that much better polynomial bounds might be achieved
with a more fine-grained analysis.

\subsection{Modulators of Dyck-grids}\label{overturning}
 
 In our proof  we  work with a simpler version of $\mathscr{D}^{(\mathsf{h},\mathsf{c})},$
 where the annulus grid is suppressed. 
The new parametric graph  is  denoted by $\widetilde{\mathscr{D}}^{(\mathsf{h},\mathsf{c})}=\langle \widetilde{\mathscr{D}}_{k}^{(\mathsf{h},\mathsf{c})}\rangle_{k\in\mathbb{N}},$ where the graph $\widetilde{\mathscr{D}}_{k}^{(\mathsf{h},\mathsf{c})}$ is obtained from $\mathscr{D}_{k}^{(\mathsf{h},\mathsf{c})}$ by removing the set $S$ of $4k^2$ vertices corresponding to the annulus grid $\mathscr{A}_{k}$
and, in the resulting graph,  among the $2k$ neighbors of the 
vertices that have been removed,
we join by edges the $k$ pairs that belong to the same cycle of $\mathscr{D}_{k}^{(\mathsf{h},\mathsf{c})}$ (see \cref{perniciously} for an example). Clearly $\widetilde{\mathscr{D}}_{k}^{(\mathsf{h},\mathsf{c})}\leq  \mathscr{D}_{k}^{(\mathsf{h},\mathsf{c})}$
and it is easy to see that ${\mathscr{D}}_{k}^{(\mathsf{h},\mathsf{c})}\leq  \widetilde{\mathscr{D}}_{2k}^{(\mathsf{h},\mathsf{c})}.$ This implies the following observation.
\begin{observation}
\label{mediocrity}
For every $(\mathsf{h},\mathsf{c})\in \mathbb{N}\times[0,2]\setminus\{(0,0)\},$  $\mathscr{D}^{(\mathsf{h},\mathsf{c})}\sim_{{\mathsf{L}}}\widetilde{\mathscr{D}}^{(\mathsf{h},\mathsf{c})}.$
\end{observation} 
\begin{figure}[ht]
  \begin{center}
  \scalebox{1.1}{\includegraphics{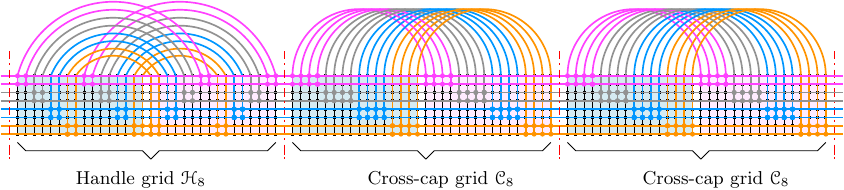}}
  \end{center}
    \caption{The graph $\widetilde{\mathscr{D}}_{8}^{1,2}$ and a half-integral packing of four subdivisions of $\widetilde{\mathscr{D}}_{2}^{1,2},$ colored by different colors.}
  \label{formalistic}
\end{figure}
Notice that there are $\mathsf{h}$ many
$(2k\times 2k)$-grids and $2\mathsf{c}$ many
 $(2k\times k)$-grids in $\widetilde{\mathscr{D}}^{(\mathsf{h},\mathsf{c})}_k$ 
 that are all pairwise disjoint and whose union spans all the vertices of  $\widetilde{\mathscr{D}}^{(\mathsf{h},\mathsf{c})}_k$.

Let $\widetilde{\mathscr{D}}^{(\mathsf{h},\mathsf{c})}_{k}$ be a graph in $\widetilde{\mathscr{D}}^{(\mathsf{h},\mathsf{c})}$.
 Given some non-negative integer $b≤k$, we recursively define the set $b$-\textsf{outer}$(\widetilde{\mathscr{D}}^{(\mathsf{h},\mathsf{c})}_k)$ of the \emph{$b$-outermost cycles} of  $\widetilde{\mathscr{D}}^{(\mathsf{h},\mathsf{c})}_k$ as follows.
We define $0$-\textsf{outer}$(\widetilde{\mathscr{D}}^{(\mathsf{h},\mathsf{c})})=\emptyset$. If $b≥1$, then we define $1$-\textsf{outer}$(\widetilde{\mathscr{D}}^{(\mathsf{h},\mathsf{c})})$ as the union of $(b-1)$-\textsf{outer}$(\widetilde{\mathscr{D}}^{(\mathsf{h},\mathsf{c})})$
and the singleton containing the  unique cycle of 
$$D\coloneqq \widetilde{\mathscr{D}}^{(\mathsf{h},\mathsf{c})}-\bigcup_{C\in \text{$(b-1)$-\textsf{outer}$(\widetilde{\mathscr{D}}^{(\mathsf{h},\mathsf{c})})$}}V(C)$$ 
that contains all  vertices of $D$ that have  degree 3. }

Let $G$ and $H$ be two graphs.
A \emph{half-integral packing of $H$ in $G$}
is a set $\mathcal{P}=\{H_{1},\ldots,H_{r}\}$
of subgraphs  of $G$
such that 
\begin{itemize}
\item for every $i\in[r],$ $H$ is a minor of $ H_{i}$
and 
\item no vertex of $G$ belongs to more than two graphs in $\mathcal{P}.$
\end{itemize}
 
 We will use the following observation (we omit the tedious proof that becomes apparent when considering the half-integral packing of $4$ copies of  $\mathscr{D}_{2}^{(\mathsf{h},\mathsf{c})}$  
 in $\mathscr{D}_{8}^{(\mathsf{h},\mathsf{c})}$  depicted in \cref{formalistic}).
 
 \begin{observation}
 \label{prerequisite}
 Let $(\mathsf{h},\mathsf{c})\in \mathbb{N}\times[0,2]$ and let $x,y\in\mathbb{N}_{\geq 1}.$
Then $\widetilde{\mathscr{D}}_{xy}^{(\mathsf{h},\mathsf{c})}$ contains a half-integral
packing of $x$ copies of $\widetilde{\mathscr{D}}_{y}^{(\mathsf{h},\mathsf{c})}.$
 \end{observation}

\paragraph{Definition of $\Sbbb_{\mathsf{h},\mathsf{c}}$.}
Let $h\in\mathbb{N}$ and $c\in[0,2]$ be integers where $(\mathsf{h},\mathsf{c})$
is different from $(0,0).$ 
We define $\Sbbb_{\mathsf{h},\mathsf{c}}$ as the set containing every surface where it is not possible to embed all graphs in $\mathscr{D}^{(\mathsf{h},\mathsf{c})}.$
Notice that $\Sbbb_{\mathsf{h},\mathsf{c}}$ contains every surface that does not contain $\Sigma^{(\mathsf{h},\mathsf{c})}.$
Due to \cref{territorial},  if $(\mathsf{h},\mathsf{c})\neq (0,0),$ then $\Sbbb_{\mathsf{h},\mathsf{c}}$ is non-empty.

\begin{lemma}
\label{antisthenes}
Let $(\mathsf{h},\mathsf{c})\in \mathbb{N}\times[0,2]\setminus\{(0,0)\}$
and let $\Sigma\in\Sbbb_{\mathsf{h},\mathsf{c}}.$
Then $\bdim_{\Sigma}(\mathscr{D}_{k}^{(\mathsf{h},\mathsf{c})})\in \Omega_{\mathsf{h},\mathsf{c}}(k^{1/5}).$
\end{lemma}

\begin{proof}
Notice that by  \cref{mediocrity} and the minor-monotonicity of $\bdim_{\Sigma}$, it is enough to prove that $\bdim_{\Sigma}(\widetilde{\mathscr{D}}_{k}^{(\mathsf{h},\mathsf{c})})\in  \Omega_{\mathsf{h},\mathsf{c}}(k^{1/5}).$
Again by  \cref{mediocrity} we know that  there is no surface in $\Sbbb_{\mathsf{h},\mathsf{c}}$ where all graphs in $\widetilde{\mathscr{D}}^{(\mathsf{h},\mathsf{c})}$ 
can be embedded. Let 
$k_{0}$ be minimum such that  $\widetilde{\mathscr{D}}^{(\mathsf{h},\mathsf{c})}_{k_0}$ is not embeddable
in any of the  surfaces in $\Sbbb_{\mathsf{h},\mathsf{c}}.$ Observe also that 
$|\Sbbb_{\mathsf{h},\mathsf{c}}|$ is bounded by some function of $\mathsf{h}$ and $\mathsf{c}$.  Clearly, $k_0$ depends exclusively on the choice of $\mathsf{h},\mathsf{c}.$
Let $X$ be a modulator of $\mathscr{D}_{k}^{(\mathsf{h},\mathsf{c})}$ to the {embeddability} in $\Sigma.$

Our target is to prove that $\bdim(\widetilde{\mathscr{D}}^{(\mathsf{h},\mathsf{c})}_{k},X)\in \Omega_{\mathsf{h},\mathsf{c}}(k^{1/5}).$

For each $b\in[0,k],$ we set $\mathcal{C}_{b}\coloneqq b$-\textsf{outer}$(\widetilde{\mathscr{D}}^{(\mathsf{h},\mathsf{c})}_k)$, and we 
define $B_{b}=V(\cupall \mathcal{C}_{b})$ and  $X_{b}=X\setminus B_b.$
The value of $b$ will be determined later.
\smallskip

\begin{figure}[ht]
\begin{center}
\includegraphics[height = 6.8cm]{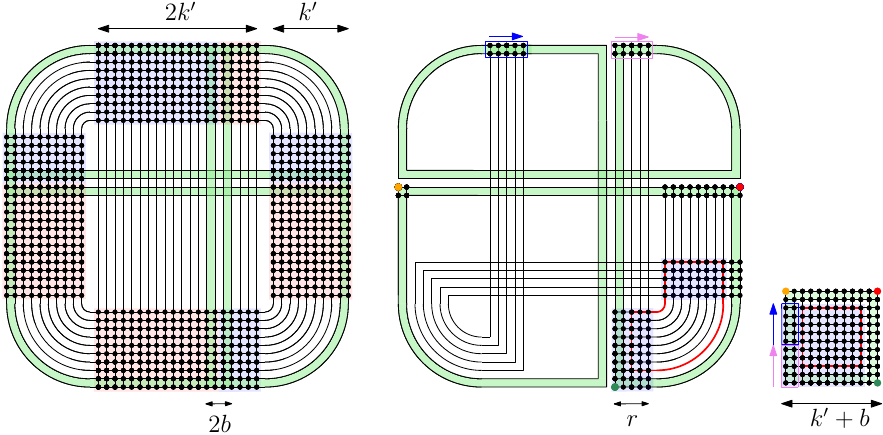}
\end{center}
  \caption{A visualization of the first subcase of Case 1 in the proof of \cref{antisthenes}, where $k'=10,$ $b=2,$ $m=12,$ and $r=5.$  The $b$ {outermost cycles} are colored \textcolor{AO}{green}.
  We depict the partition of $G^{\star}_i.$ to four pairwise disjoint $(2k'\times k')$-grids, and the $(m\times m)$-grid $\Gamma_{m}$ as a minor in it. The perimeter of the central $(m-2b) \times (m-2b)$-sub-grid is depicted \red{red}. The \red{red} and the \violet{violet} squares, as well as the attached arrows indicate the correspondence between vertices.}
\label{nesickness}
\end{figure}

By \cref{prerequisite}, there is a set $\mathcal{P}_b$ of $\lfloor \frac{k-b}{k_0}\rfloor$ subgraphs of $\widetilde{\mathscr{D}}^{(\mathsf{h},\mathsf{c})}_{k}$ that is a half-integral packing of $\widetilde{\mathscr{D}}^{(\mathsf{h},\mathsf{c})}_{k_0}$ in $\widetilde{\mathscr{D}}^{(\mathsf{h},\mathsf{c})}_{k}$ and such that the graphs in $\mathcal{P}_b$ do not contain vertices in $B_{b}.$
Notice that $X_{b}$ should contain at least one vertex from every graph in $\mathcal{P}_b$ and no vertex of $X_{b}$ can belong to more than two graphs in $\mathcal{P}_b.$
This implies that 
\begin{eqnarray}
    |X_{b}|\geq  \lfloor \frac{k-b}{2k_0}\rfloor. \label{this_zer}
\end{eqnarray}
Observe now that the vertex set of $\widetilde{\mathscr{D}}^{(\mathsf{h},\mathsf{c})}_{k}$ can be partitioned into $h+c$ vertex sets $V_{1},\ldots,V_{h+c}$ such that, for each $i\in[h+c],$ $\widetilde{\mathscr{D}}^{(\mathsf{h},\mathsf{c})}_{k}$ contains 
 $G_{i}$ as a $V_{i}$-minor where ${G}_{i}$ is the copy of either $\mathscr{H}_{k}$ or $\mathscr{C}_{k}$ and $V_{i}=V({G}_{i})$ (this partition is indicated in \cref{formalistic}).
As before, we define $\mathcal{C}_{b}^i$ to be the set of the $b$ {outermost} cycles of $G_{i}$, and we define $B_{b}^i\coloneqq V(\cupall \mathcal{C}_{b}^{i})$ and  $X_{b}^{i}=X\setminus B_b^i.$ We now choose $i$ such that $|X_{b}^{i}|$ is maximized and  this means, because of \eqref{this_zer}, that  
\begin{eqnarray}
|X_{b}^i|\geq \lfloor \frac{k-b}{2k_0(h+c)}\rfloor.\label{first_one}
\end{eqnarray}
We distinguish two cases, depending on whether ${G}_{i}=\mathscr{H}_{k}$ or ${G}_{i}=\mathscr{C}_{k}.$
For simplicity we drop $i$ in the notation of $X_{b}^i$, as $i$ is fixed from now on.
\medskip

\noindent\emph{Case 1}. ${G}_{i}=\mathscr{H}_{k}.$ Let $k'=\lfloor k/2\rfloor.$
Let $\overline{\Ccal}_{b}^{i}$ be the cycles of $\mathscr{H}_{k}$ that are not in $\mathcal{C}_{b}^i$.
As a first step, among the $k-b$ cycles of $\overline{\Ccal}_{b}^{i}$ we consider the set $\mathcal{C}$ of $k'-b$ such cycles whose union meets the biggest possible number of vertices in $X_{b}.$
Observe that the cycles of $\mathcal{C}$ meet at least  $|X_{b}|(\frac{1}{2}-\frac{b}{2k-2b})$ vertices of $X_{b}.$ 
Let $X_{b}^{\star}=X_{b}\cap V(\cupall \mathcal{C})$ and, using \eqref{first_one}, observe that 
\begin{eqnarray}
|X_{b}^{\star}|\geq \lfloor (\frac{1}{2}-\frac{b}{2k-2b})\frac{k-b}{2k_0(h+c)}\rfloor\geq \lfloor \frac{k-2b}{4k_0(h+c)}-\frac{b}{4k_0(h+c)}\rfloor=\lfloor \frac{k-3b}{4k_0(h+c)}\rfloor.\label{second_one}    
\end{eqnarray}

We now consider the minor $G^{\star}_i$ of 
${G}_{i}$ that is obtained by contracting every edge that is not in a transaction,  
has endpoints in different cycles of $\overline{\Ccal}_{b}^{i}$ 
and  do not have both endpoints in $V(\cupall \mathcal{C})$ (we assume that, during contractions, vertices in $X^{\star}_{b}$ prevail).
Notice that  $G_{i}^{\star}$ is a minor of $G$ containing the vertices of $X_{b}^{\star}$
(that are also vertices of $G_{i}$).

Observe  that $G^{\star}_i$ contains a $(8k' \times  k')$-cylindrical grid
as a spanning subgraph that contains $k'$ cycles and $8k'$ paths that we call
 \emph{rails}. Anti-diametrical pairs of rails can be seen as 
 subpaths of paths of length $2k'-1,$ we call them  \emph{crossing paths}, whose endpoints are in the outer cycle of  $G^{\star}_i.$
 Notice that, the vertex set of   $G^{\star}_i$ 
 is the union of two disjoint $2k'\times 2k'$-grids, say $\Gamma^1$ and $\Gamma^{2}.$ We call a crossing path of $G^{\star}_i.$ 
\emph{vertical} if it belongs to $\Gamma^{1}$
 and  \emph{horizontal} if it belongs to $\Gamma^{2}.$
 
Also, $G^{\star}$ contains four pairwise disjoint $(2k'\times k')$-grids.
We call them \emph{north, south, west}, and \emph{east grid}.
(In Figures \ref{nesickness} and \ref{echibition} each these grids is bicolored by \red{red}/\blue{blue} frames.)

We now assume, without loss of generality, that the south grid has the majority of the 
vertices of $X_b^{\star}.$ Let this vertex set be  $\widehat{X}_{b}^{\star}$ and observe, using \eqref{second_one},  that 
\begin{eqnarray}
|\widehat{X}_{b}^{\star}|\geq  \lfloor |{X}_{b}^{\star}|/4\rfloor\geq  \lfloor \frac{k-3b}{16k_0(h+c)}\rfloor.\label{one_tho_g}    
\end{eqnarray}

Next we observe that there 
is a subset $S_{b}$ of $\widehat{X}_{b}^{\star}$ and a set $\mathcal{Q}$ of $2b$ consecutive  {vertical crossing paths} 
such that  $|S_{b}|\geq  \lfloor\frac{|\widehat{X}_{b}^{\star}|}{2b}\rfloor-1$ and $S_{b}\cap \cupall\mathcal{Q}=\emptyset.$
W.l.o.g. we assume that the majority of the 
vertices of $S_{b}$ is ``on the right''  of the vertical crossing paths in $\mathcal{Q}$, and we denote these vertices by  $S_{b}'.$
Clearly, using \eqref{one_tho_g},  
\begin{eqnarray}
|S_{b}'|\geq  \lfloor |S_{b}|/2\rfloor\geq  \lfloor \frac{k-3b}{64bk_0(h+c)}\rfloor-1.\label{more_olio}
\end{eqnarray}

\begin{figure}[ht]
\begin{center}
\includegraphics[height = 6.8cm]{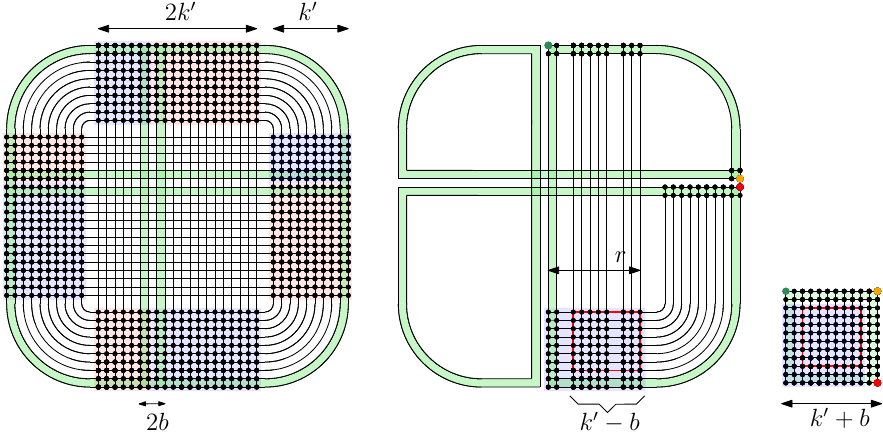}
\end{center}
  \caption{
  A visualization of the second subcase of Case 1 in the proof of \cref{antisthenes}, where $k'=10,$ $b=2,$ $m=12,$ and $r=13.$ The $b$ {outermost cycles} are colored \textcolor{AO}{green}.
  We depict the partition of $G^{\star}_i.$ to four pairwise disjoint $(2k'\times k')$-grids, and the $(m\times m)$-grid $\Gamma_{m}$   as a minor in it. The perimeter of the central $(m-2b) \times (m-2b)$-subgrid is depicted \red{red}.}
\label{echibition}
\end{figure}

Let  $r$ be the number of consecutive {rails} in the south grid starting from the second half of the paths in $\mathcal{Q}$ (notice that  $r\geq  b+1$ as at least one rail should be intersected by $S_{b}'$). In \cref{nesickness} and \cref{echibition} we have that $r=5$ and $r=13$ respectively.
Let $m=k'+b.$
We will consider a {subdivision} of the $(m\times m)$-grid $\Gamma_{m}$ by considering rails of the south grid and, possibly, some of the east one. For this,  we distinguish two subcases as follows:

\begin{itemize}
\item If $r<k'$ then we complete the $r$ rails with the $k'-r+b$ lower rails in the east grid and construct the $(m\times m)$-grid $\Gamma_{m}$ (see \cref{nesickness}). Set $S_{b}''\coloneqq S_{b}'.$

\item If $r\geq  k,$ then among all $((k'-b)\times k)$-grids that can be made 
using $k'-b$  rails of the south grid and on the right of 
$\mathcal{Q},$ consider those that intersect the biggest number of vertices in $S_{b}'.$ Also complete the  $(m\times m)$-grid $\Gamma_{m}$  by adding $b$ rails from the rightmost grid (see \cref{echibition}).
Let  $S_{b}''$ be the vertices of $S_{b}'$ 
that belong to the $k'-b$ chosen  rails 
and observe that $|S_{b}''|\geq  \lfloor |S_{b}'|/2\rfloor.$
\end{itemize}
Let $f(b,k)\coloneqq \lfloor \frac{k-3b}{128bk_0(h+c)}\rfloor-1$. 
Notice that, in any case, 
(recalling \eqref{more_olio}), it holds that 
\begin{eqnarray}
|S_{b}''|\geq  \lfloor |S_{b}'|/2\rfloor\geq  f(b,k)\text{}   
\end{eqnarray}

\begin{figure}[ht]
\begin{center}
\includegraphics[height = 6.6cm]{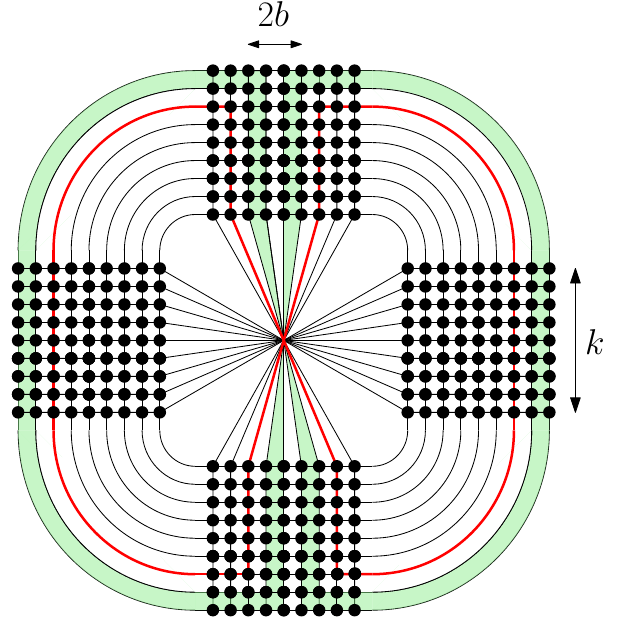}
\end{center}
  \caption{
  A visualization of Case 2 in the proof of \cref{antisthenes}, where $k=9,$   $b=2,$ and  $m=18.$ The $b$ {outermost cycles} are colored \textcolor{AO}{green}.
  We depict the partition of ${G}_{i}=\mathscr{C}_{k}$ to four pairwise disjoint $(k\times k)$-grids and the $(m\times m)$-grid as a spanning subgraph of $\mathscr{C}_{k}.$ The perimeter of the central $\big((m-2b) \times (m-2b)\big)$-subgrid is depicted \red{red}.}
\label{scrupulously}
\end{figure}

In other words, the central $(m-2b) \times (m-2b)$-subgrid of this $(m\times m)$-grid    contains at least $f(b,k)$ vertices whose precursors contain vertices from $X_{b}.$
\bigskip

\noindent\emph{Case 2}. ${G}_{i}=\mathscr{C}_{k}.$
As we did before, we refer to the $4k$ paths of $\mathscr{C}_{k}$ as \emph{rails}. Also, anti-diametrical pairs of rails are seen as subpaths of $2k$ paths of length $2k-1$ and, as before, we call them  \emph{crossing paths} (see 
 Figure \ref{scrupulously}). 
Notice that the vertex set of  $\mathscr{C}_{k}$ is the union of two disjoint   $(2k\times k)$-grids, say $\Gamma^1$ and $\Gamma^{2}.$ We call a crossing path of $\mathscr{C}_{k}$ \emph{vertical} if it belongs to $\Gamma^{1}$ and \emph{horizontal} if it belongs to $\Gamma^{2}.$

Observe that there is a subset $S_{b}$ of $X_{b}^i$ and a set $\mathcal{Q}$ of $2b$ consecutive  crossing paths 
such that  $|S_{b}|\geq  \lfloor\frac{|X_b^i|}{2b}\rfloor-1$ and $S_{b}\cap \,\cupall\mathcal{Q}=\emptyset.$
We have that $|S_{b}|\geq  \lfloor \frac{k-3b}{4bk_0(h+c)}\rfloor-1≥f(b,k).$

Let $m=2k.$
Observe that  $\mathscr{C}_{k}$ contain as  subgraph an $(m\times m)$-grid $\Gamma_{m}$ spanning all of its vertices.
To see this, it is enough to remove all the edges between the vertices of the $b$-th and the vertices of the $(b+1)$-th path of $\mathcal{Q}$ from $\mathscr{C}_{k}$.

Notice that the subgraph of $\Gamma_{m}$ induced by its $b$ {outermost} vertices is exactly $\Gamma_{m}\setminus (B_{b}\cup \cupall\mathcal{Q}).$
Moreover, the central  $\big((m-2b)\times (m-2b)\big)$-subgrid of  $\Gamma_{m}$ contains at least $f(b,k)$ vertices from $X_b.$
\bigskip

In both of the two cases above, we choose the margin $b_k$ as the maximum 
integer $b$ such that $\lfloor \sqrt[4]{f(b,k)}\rfloor\geq  2b.$
This implies that  $\lfloor\sqrt[4]{|S_{b_k}|}\rfloor\geq  2b_{k},$ therefore, by \cref{subjection}, $\bdim({\widetilde{\mathscr{D}}^{(\mathsf{h},\mathsf{c})}_{k}},X)\geq  \lfloor\sqrt[4]{|S_{b_k}|} \rfloor\geq  2b_{k}.$   
Observe now that in both cases $b_{k}\in Θ_{\mathsf{h},\mathsf{c}}(k^{1/5}),$ therefore, by \cref{imperiously}
 $\bdim({\widetilde{\mathscr{D}}^{(\mathsf{h},\mathsf{c})}_{k}},X)\in \Omega_{\mathsf{h},\mathsf{c}}(k^{1/5}).$
\end{proof}

\subsection{Lower bounds under the presence of clique-sums}
\label{righteousness}

It now remains to prove that, given that $\bdim_{\Sigma}(\widehat{\mathscr{D}}_{k}^{(\mathsf{h},\mathsf{c})})$ is big, then also $\bdim_{\Sigma}^{\star}(\widehat{\mathscr{D}}_{k}^{(\mathsf{h},\mathsf{c})})$ is also big. For this we need to introduce some new concepts and some preliminary results.

A \emph{partially triangulated $(k\times k)$-grid} is any graph that can be obtained from a $(k\times k)$-grid (we call it the {\em underlying grid}) by adding edges in a way that the resulting graph remains planar.
\\

We say that a graph $G$ with a vertex set $X$, can 
be \emph{$X$-contracted} to a graph $H$ 
if $H$ is an $X$-minor of $H$, certified 
by some collection $\mathcal{S}=\{S_{v}\mid v\in V(H)\}$ of pairwise vertex-disjoint {connected} vertex sets, as defined in \cref{recommendation}, with the additional 
demand that for every edge of $G$ between vertices 
belonging  in two distinct sets $S_{v}$ anf $S_{u}$, there is an edge between $v$ an $u$ in $H$.

 The following lemma follows easily from the results of \cite{KawarabayashiTW18anewp} (see also \cite{SauStamoulisThilikos2024}). We provide a sketch of the proof as the arguments are quite standard  and a detailed description would require
lengthy definitions  of concepts that do not offer anything in the rest of this paper.

%The following result follows easily from the results of  \cite{FominGT11contr}, in particular it is  a direct byproduct of the proof of \cite[Lemma 11]{FominGT11contr}.
%The polynomial bound in $\Ocal(\big(\hw(G)\big)^{24})$ it can be easily derived by the main result of \cite[Theorem 1.5]{KawarabayashiTW18anewp}, taking into account that there is a constant $c'$
%such that every $2r$-wall can be contracted to a  partially triangulated  $(r\times r)$-grid (see the proof of \cite[Lemma 5]{FominGT11contr}).

%
%
%

%
%\begin{proposition}
%\label{outrisiop}
%There is a universal constant $c$ such that  
%every connected graph $G$ where $\bg(G)$$\geq  c\cdot  \big(\big(\hw(G)\big)^{48}+r\big)$ contains a vertex set $A$ of size $c\cdot \big(\hw(G)\big)^{24}$
%such that $G-A$ can be contracted to a dusted partially triangulated 
%$(r\times r)$-grid. 
%\end{proposition}

\begin{proposition}
\label{rebellious}
There is a polynomial function $f_{\ref{rebellious}}: \Nbbb^{3}\to\Nbbb$ such that 
if $G$ is a connected graph and $X\subseteq V(G)$
where  $\bdim(G,X)≥f_{\ref{rebellious}}(\hw(G),\Delta(G),r)$, then  $G$ 
can be $X$-contracted to a partially triangulated 
$(r\times r)$-grid. 
\end{proposition}

\begin{proof}[Proof (sketch)]
From \cite[Theorem 1.5]{KawarabayashiTW18anewp} (see also \cite{SauStamoulisThilikos2024})
we known that there are polynomial functions $f_{1}:\Nbbb^{2}\to\Nbbb$ and $f_{2}:\Nbbb\to\Nbbb$
such that if $G$ contains a $f_{1}(k,\hw(G))$-wall $W$ as a subgraph then 
there is a set $A\subseteq V(G)$, $|A|≤f_{2}(\hw(G))$ and a subwall $W'$ 
of $W$ such that $W'$ is a ``flat'' $k$-wall in $G-A$.
We also know that the $≤f_{2}(\hw(G))$ vertices in $A$ have at most $f_{2}(\hw(G))\cdot\Delta(G)$
neighbors in $G-A$ which, in turn implies  that, if we set $k\coloneqq 8r\lceil \sqrt{f_{2}(\hw(G))\cdot\Delta+1}\rceil$,
then there is a subwall $W''$ of $W'$ that is a flat $8r$-wall in $G$.

By \cref{uncoordinated} we know that 
if  $\bdim(G,X)≥ck$ then  $G$ contains 
a rooted $k$-wall $(W^+,X)$  as a subgraph.
Let $W$ be the $k$-wall obtained if we remove from $W$ its hairs.
By applying the conclusion of the previous paragraph on $W$,
we  may find a subwall $W''$ that is a flat $8r$-wall of $G$.
We next attach 
to each maximally induced path  of $W''$ one of the hairs from $W^+$.  The flatness of  the $8r$-wall 
$W''$ and the existence of a hair attached to every maximally induced path of $W$
permits to contract edges in $G$ so to contract 
$W''$ to a partially triangulated  $(r\times r)$ where each vertex has a precursor in $X$. This can be done by using the concept of ``canonical partitions'' as in \cite[Subsection 3.1]{SauST23kapex}.
\end{proof}

For $k\in\mathbb{N}_{{\geq 1}},$ a graph is a \emph{crossed $k$-grid} if it can be obtained as follows:
first take the $\big((k+2)\times (k+2)\big)$-grid, then remove all outermost layer  edges (edges with both endpoints  of degree $<4$), subdivide each of the remaining edges at least once and then take the line graph of the resulting graph, i.e.,
correspond a vertex to each edge and make two vertices adjacent if the corresponding edges share an endpoint (see \cref{liquidating}).

\begin{observation}
\label{identifies}
If $G$ has a tree decomposition where all torsos have Hadwiger number at most $h,$ then also $G$ has Hadwiger number at most $h.$
\end{observation}

\begin{lemma}
\label{commoditized}
There is some universal constant $c_{\ref{commoditized}}$ such that 
if a graph $G$ contains a crossed $t$-grid as a subgraph, for some $t\in\mathbb{N}_{\geq 1},$ then $t\leq  c_{\ref{commoditized}}\cdot (\hw(H))^2.$
\end{lemma}

\begin{proof}
If, in a crossed $2t(t-1)$-grid, we contract all edges that do not belong to $4$-cliques we obtain a graph that contains the grid-like graph $H^1_{t(t-1)}$ defined\footnote{Let $r ≥ 1$ be an integer, and let $H_{2r}$ be the $(2r \times 2r)$-grid with vertex set $[2r] \times [2r]$. The graph $H^1_{2r}$ is defined as the graph obtained from $H_{2r}$  by adding all edges
with endpoints $(i, r)$ and $(i + 1, r + 1)$, and all edges with endpoints $(i, r + 1)$ and $(i + 1, r)$ for all
$i\in[2r-1]$.  In other words, $H^{1}_{2r}$ is constructed from the $(2r \times  2r)$-grid by adding a pair of crossing edges in each face of the middle row of faces.}.
by Kawarabayashi, 
Thomas, and Wollan in \cite{KawarabayashiTW18anewp} as a subgraph.
According to \cite[Lemma 3.2.]{KawarabayashiTW18anewp}, $H^1_{t(t-1)}$ contains $K_{t}$ as a minor.
\end{proof}

Given a graph class $\Gcal$, we define $$\bdim_{\Gcal}(G)=\min\{\bdim(G,X)\mid X\subseteq V(G \mbox{~and~}G-X\in \Gcal\}.$$
\begin{lemma}
\label{neutrality}
Let ${\cal G}$ be a graph class of Hadwiger number at most $η$.
Then, for every graph $G$,  $\hw(G)\leq (\bdim_{\Gcal}(G))^2+η+1$.
\end{lemma}
\begin{figure}[ht]
\begin{center}
\includegraphics[height = 6.5cm]{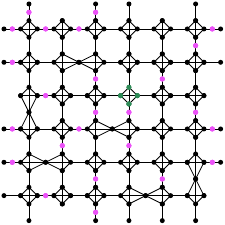}
\end{center}
  \caption{A crossed $6$-grid. Subdivision vertices are depicted in \textcolor{DeepMagenta}{magenta}.}
  \label{liquidating}
\end{figure}
\begin{proof}
Suppose that $\hw(G)= t.$
Let $X$ be a modulator of $G$ to $\mathcal{G}$. 
Notice that $G$ can be identified, via contractions, to the clique $K_{t}.$ Let $X'$ be the {heir} of  $X$ during these contractions. Notice that  $X'$ remains a modulator of $K_{t}$  to $\mathcal{G},$ therefore $|X'|\geq  t-η.$
Notice that $\bdim(G,X)\geq  \bdim(K_{t},X')$ and that $\bdim(K_{t},X')\geq  \bdim(K_{t}[X'])\geq   \lfloor |X'|^{1/2}\rfloor \geq   \lfloor(t-η)^{1/2}\rfloor\geq(t-η-1)^{1/2}.$ This implies that $\bdim_{\Gcal}(G)\geq  (t-η-1)^{1/2}$, and the lemma follows.
\end{proof}

\paragraph{A new variant of the Dyck-grid.}
For the purposes of our proofs, we will consider an ``enhanced version'' of $\widetilde{\mathscr{D}}_{k}^{(\mathsf{h},\mathsf{c})},$ namely $\widehat{\mathscr{D}}^{(\mathsf{h},\mathsf{c})}=\langle \widehat{\mathscr{D}}_{k}^{(\mathsf{h},\mathsf{c})}\rangle$ which is obtained from $\widetilde{\mathscr{D}}_{k}^{(\mathsf{h},\mathsf{c})}$ as follows:
Consider a set of $k(h+c)$ pairwise disjoint paths on $4$ vertices of the {outermost cycle}.
We take them so that the paths are parts of the  perimeters of the (pairwise disjoint) $2h$ $(2k\times k)$-grids and $c$  $(2k\times 2k)$-grids
corresponding to the handles and the crosscaps of 
$\widetilde{\mathscr{D}}_{k}^{(\mathsf{h},\mathsf{c})}$ {(there is a unique way these paths can be  considered).}
Then, for each of these paths,  introduce a new vertex and make it adjacent with its four vertices. We call the $k(h+c)$ new vertices in this construction {\em satellite} vertices (see \cref{tolerantly}).

A \emph{separation} of a graph $G$ is a pair $(B_{1},B_{2})$
such that there is no edge from $B_{1}\setminus B_{2}$ to $B_{2}\setminus B_{1}$.
We say that a parametric graph $\mathscr{H}=\langle \mathscr{H}_k \rangle_{k\in \Nbbb}$ is \emph{$f$-tightly connected} for some non-decreasing function $f:\Nbbb\to\Nbbb$, if 
for every separation $(B_1,B_2)$ of $\mathscr{H}_{k}$ of order $q<k$ such that both $G[B_{1}\setminus B_{2}]$ and $G[B_{2}\setminus B_{1}]$ are connected, it holds that one of $B_1,$ $B_2$ has at most $f(q)$ vertices.

\begin{lemma}
\label{letas_s}
Let $\p:\Gall\to\Nbbb$ be some minor-monotone graph parameter and let $h:\Nbbb\to\Nbbb$ such that for every graph $G$, $\hw(G)\leq h(\p(G))$.
Let also  $\mathscr{H}=\langle \mathscr{H}_k \rangle_{k\in \Nbbb}$ be a parametric graph
that is $g$-tightly connected, for some $g:\Nbbb\to\Nbbb$.
Let  $r\in\Nbbb$ and let ${(T,\beta)}$  be  a tree decomposition of $G=\mathscr{H}_k$ where, for each torso $G_{t}$, $t\in V(T)$
it holds that $\p(G_{t})\leq r$.
If $2(g(h(r)))<|V(\mathscr{H}_k)|$ and $h(r)<k$, then 
$G$ admits a tree decomposition $(\widehat{T},\widehat{\beta})$ where 
$\widehat{T}$ is a star with center $d$, where $\p(G_{d})\leq r$  ($G_{d}$ is the torso of $(\widehat{T},\widehat{\beta})$ on $d$) and where for every  $e=dd''\in E(\widehat{T})$, $\mathscr{H}_k[\widehat{β}(d')\setminus \widehat{β}(d)]$ is a connected graph and $|\widehat{β}(d')|\leq g(h(r))$.
\end{lemma}

\begin{proof}
Let $e=tt'\in E(T)$ and let $T_t$ and $T_{t'}$ be the two connected components of $T-e,$ where $t\in V(T_t)$ and $t'\in V(T_{t'}).$
We define $B_{t}\coloneqq\bigcup_{s\in V(T_{t})}β(s)$, $B_{t'}\coloneqq\bigcup_{s\in V(T_{t'})}β(s)$, and $A_{t,t'}\coloneqq B_{t}\cap B_{t'},$ i.e., $A_{t,t'}$ is the common adhesion of $t$ and $t'.$ We next modify $(T,β)$ as follows.
\medskip

\noindent\emph{Claim.} We can assume that, for every $e=tt',$ every graph in $\{G[B_{t}\setminus B_{t'}],G[B_{t'}\setminus B_{t}]\}$ is  void or connected.\smallskip

\noindent\emph{Proof of Claim}:  
Suppose that this is not the case and assume that for some edge $e=tt'\in E(T),$  $G'=G[B_{t}\setminus B_{t'}]$ is not connected and has $r$ components $H_{1},\ldots,H_{r}.$
Let $Y$ be the component of $T-tt'$ containing $t.$
We replace this component in $T$ with $r$ copies of $Y,$ say $Y_{1},\ldots,Y_{r},$ and for each one of them make the leaf $t^{i}$ corresponding to $t$ adjacent to $t'.$
The new tree is denoted by $\widehat{T}.$
We define $\widehat{β}\colon V(\widehat{T})\to 2^{V(G)}$ such that, if $s\not\in V(Y),$ then $\widehat{β}(s)\coloneqq β(s),$ and if $s$ belongs to $Y_{i},i\in[r],$ then we set $\widehat{β}(s)\coloneqq β(s)\cap V(H_{i}).$
It is easy to see that $(\widehat{T},\widehat{β})$ is a treedecomposition of $G.$
Notice that each torso $G_{\widehat{t}}$ of $(\widehat{T},\widehat{β})$ is the subgraph of  some torso $G_{t}$ of $(T,β).$
This means that if $\p(G_{t})\leq r$ holds for each torso $G_{t}$ of 
$(T,β),$ then the same holds for the torsos of  $(\widehat{T},\widehat{β})$.
By repeating the above improvements as long as there are edges violating the statement of the claim, we finally have the required tree decomposition.\medskip

Let $tt'\in E(T)$.
The set $A_{t,t'}$ induces 
a clique in $G_{t}$, therefore, $|A_{t,t'}|\leq \hw(G_{t})\leq h(\p(G_{t}))\leq h(r)<k$.
As $G$ is $g$-tightly connected, it holds hat  $\min\{|B_{t}|,|B_{t'}|\}\leq  g(|A_{t,t'}|)\leq g(h(r)),$ therefore $$\max\{|B_{t}|,|B_{t'}|\}\geq  |G|-g(h(r))> g(h(r))\geq \min\{|B_{t}|,|B_{t'}|\}.$$
This permits us to orient  every edge $tt'$ towards $t,$ if $|B_{t}|>|B_{t'}|$ and towards  $t',$ if $|B_{t'}|>|B_{t}|.$ 
This orientation of $T$ implies that there is a unique node $d$ where all incident edges are oriented towards it.
This means that, for every $t'\in N_{T}(d)$, we have that $|B_{t'}| \leq    g(h(r)).$

We now build the claimed tree decomposition $(\widehat{T},\widehat{\beta})$
where $\widehat{T}$ is the subtree of $T$ 
that is induced by $d$ and any of its neighbors $t'$ in $T$ where $B_{t'}\setminus A_{d,t'}\neq\emptyset$. Finally, for each $t'\in V(\widehat{T}\setminus\{d\})$, we set $\widehat{β}(t')=B_{t'}$.
\end{proof}

\begin{figure}[ht]
  \begin{center}
  \scalebox{1.11}{\includegraphics{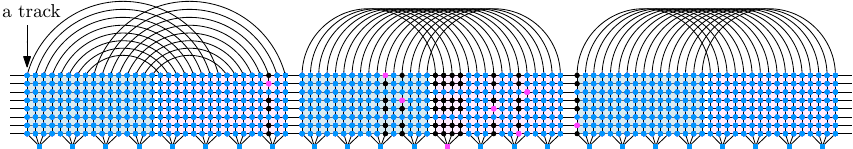}}
  \end{center}
    \caption{The graph $\widehat{\mathscr{D}}_{8}^{1,2}$. The satellite vertices are depicted as squares. The \textcolor{DeepMagenta}{magenta} and the \textcolor{CornflowerBlue}{blue} vertices correspond to the vertices in $A$ and $Y$ respectively, in the proof of \cref{vegetables}.}
  \label{tolerantly}
\end{figure}

\begin{lemma}
\label{vegetables}
Let $(\mathsf{h},\mathsf{c})\in \mathbb{N}\times[0,2]\setminus\{(0,0)\}$.
Then $\widehat{\mathscr{D}}^{(\mathsf{h},\mathsf{c})}$ is $f_{\ref{vegetables}}$-tightly connected for $f_{\ref{vegetables}}(q)=(2q+1)^2$.
\end{lemma}

%\medskip
%\hrule
%\red{\bf 12 of May, 16:03}
%\hrule
%\medskip

%
%
%
\begin{proof}
Let $G\coloneqq \widehat{\mathscr{D}}_{k}^{(\mathsf{h},\mathsf{c})}.$ 
Let $(B_1,B_2)$ be a separation of $G$ of order $q<k.$
Let  $A=B_{1}\cap B_{2},$ $D_{1}=B_{1}
\setminus A,$ and $D_{2}=B_{2}\setminus A.$ Also we assume that both $G[D_{1}]$ and  $G[D_{2}]$ are connected.
Let $Q$ be the spanning $(4k(h+c)\times  k)$-annulus grid of $G.$

Clearly $Q$ contains $k$ cycles. Also it contains  $4k(h+c)$ paths, each  on $k$ vertices, which we call {\emph{tracks}} (these paths are the ``vertical'' paths as the one depicted by an arrow in \cref{tolerantly}.).
We also use $R$ for the set of the satellites of $\widehat{\mathscr{D}}_{k}^{(\mathsf{h},\mathsf{c})}$.
We  enhance the tracks by extending each of them to the unique satellite vertex that is adjacent to one of its {endpoints}.
Let $Y$ be the union of all the $\geq  k-q$ cycles that are \textsl{not} met by $A$ and 
of all $\geq 4k(h+c)-4q$ tracks that are \textsl{not} met by $A$ (observe that  if a vertex of $A$ belongs to $R$, then it meets four tracks). See \cref{tolerantly} for a visualization of $A$ and $Y$ in the graph $\widehat{\mathscr{D}}_{8}^{1,2}$.

As every track in $Y$ has a common endpoint to every cycle in $Y,$  we obtain that $Y$ is connected, therefore $V(Y)$ is either a subset of $D_{1}$ or a subset of $D_{2}.$
W.l.o.g., we assume that $V(Y)\subseteq D_{1}.$ 
We next prove that $|B_{2}|\leq  (2q+1)^2.$

Let $x$ be a vertex of $D_{2}$
and let $P_x$ be some path of $G[B_{2}]$ starting from $x$ and finishing 
to some vertex of $A$ and with all internal vertices in $D_{2}.$
This path cannot meet more than $q=|A|$ cycles of  $Q$ because  each such cycle contains some vertex of $S\subseteq D_{1}.$ Similarly, $P_{x}$ cannot meet more than $q=|A|$  tracks as each such track contains some vertex of  $Y\subseteq D_{1}.$
This implies that every vertex of $G[B_{2}]$ should be at distance 
at most $q$ from $x.$
It is now easy to verify that, in $G,$ the vertices at distance at most $q$
from some $x\in A$ is upper bounded by $(2q+1)^2.$
As all vertices of $B_{2}$ are accessible from $x$ at this distance in $G[B_{2}],$ we conclude that   $|B_{2}|\leq  (2q+1)^2.$
\end{proof}

\begin{lemma}\label{blackbox} 
Let $G$ be a 4-connected  graph of maximum degree $\Delta$.
Let $r\in\Nbbb$ and let $(T,\beta)$ be a tree decomposition of $G$ where $T$ is a star
with center $t$ and such that for every  $e=tt'\in E(T)$, $G[β(t')\setminus \beta(t)]$ is a connected graph on at most $l$ vertices. Let $G_{t}$ be the torso of $(T,\beta)$ at $t$ and 
we denote  $G^c=(V(G),E(G)\cup E(G_{t}))$, $m=\hw(G^c)$, 
and $B=\bigcup_{t'\in V(T)\setminus\{t\}}β(t')$.
There is a function $f_{\ref{blackbox}}:\Nbbb^4\to\Nbbb$ such that
if  $X$ is a subset of $V(G_t)$ where $\bdim(G_t,X)\leq   r$, then  $\bdim(G^c,X\cup B)\leq  f_{\ref{blackbox}}(r,m,l,\Delta)$.
Moreover, $f_{\ref{blackbox}}(r,m,l,\Delta)\in$  $\poly(m+l+r+\Delta)$.
\end{lemma}

\begin{proof}
We need to introduce  some constants (all depending on $\Delta$, $m$, $l$, and $r$).  \cref{rebellious}  \cref{commoditized} . 
We set 
\begin{align*}
\widehat{h}  &\coloneqq 2l+1, \\
d & \coloneqq (c_{\ref{commoditized}}m^2+3)^4\cdot (1+2l)^2, \\
t & \coloneqq 25d,\\
a & \coloneqq \lceil{(t+1)^{1/2}}\rceil\cdot (r+1)+4+2\widehat{h}  \mbox{~and}\\
a' &\coloneqq f_{\ref{rebellious}}(m,\Delta,a)-1\in \poly(m+l+r+\Delta)\end{align*}

We set $f_{\ref{blackbox}}(r,m,l,\Delta)\coloneqq a'$.
Assume that  $\bdim(G^c,X\cup B)>a'$, therefore, 
$G^{\mathsf{c}}$  contains an $(\big(a'+1)\times(a'+1)\big)$-grid as a $(X\cup B)$-minor. We claim that  it is enough to prove that  
\begin{eqnarray}
\bdim(G^{\mathsf{c}}-B,X\setminus B) & > & r.\label{eq_all_trages}
\end{eqnarray}
Indeed, as  $G^{\mathsf{c}}-B$ is a subgraph  of $G_t$, this implies   {$\bdim(G_t,X)>  r$}, which yields the lemma.

{The rest of the proof is dedicated to the proof of \eqref{eq_all_trages}.}

\begin{figure}[ht]
  \begin{center}
  \scalebox{1}{\includegraphics{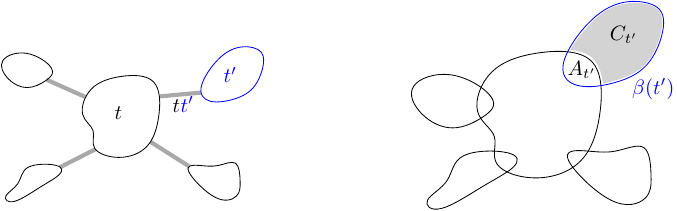}}
  \end{center}
    \caption{The star decomposition $(T,\beta)$. Left: the tree $T$ with the central vertex $t$ of $T$  and some neighbor $t'$ of $t$. Right: the graph $G$ (or $G^{\mathsf{c}}$)
    and the sets $A_{t'}$, $C_{t'}$, and $\beta(t')$ corresponding to the vertex $t'$.}
    
  \label{tolersantly}
\end{figure}

By \cref{rebellious}, we have that $G^{\mathsf{c}}$  can be $(X\cup B)$-contracted to a  dusted partially triangulated $(a\times a)$-grid $D_{a}.$ 
Keep in mind that the fact that $D_{a}$ is a contraction of $G^{\mathsf{c}}$
means that if two vertices of $D_{a}$ are at  distance $\geq \delta $ in $D_{a}$, then 
their precursors in $G^{\mathsf{c}}$ are also at distance $\geq \delta $.

Let $q=a-2\widehat{h}$ and consider  the central  $(q\times q)$-grid $D_{q}$ of $D_{a}$
We say that a vertex $v$ of  $D_{q}$ is 
\emph{marked} by  some vertex $t'\in N\coloneqq V(T)\setminus\{t\},$  if its {precursor}  in
$H$, and also in $G^\mathsf{c}$, contains {some} 
{vertex}  of $C_{t'}\coloneqq \beta(t')\setminus β(t)$. 
Notice that $v$ may be marked by more than one vertex $t'$ of $N.$

We  set, for each $t'\in N,$  $A_{t'}=β(t)\cap β(t')$. 
As an important ingredient of the proof, we prove that there cannot be many marked vertices.

\begin{claim}\label{claimmarked}
$D_{q}$ has at most  $d$  marked vertices.
\end{claim}

\begin{cproof}
Let ${M}$ be the set of   marked
vertices of $D_{q}$ and suppose to the contrary that 
\begin{eqnarray}
    |{M}|>d=(c'm^2+3)^4\cdot (1+4l)^2.\label{more_one}
\end{eqnarray}

We say that two  marked vertices $v_{1}$ and $v_{2}$ of $D_{q}$ are \emph{independent}
if there is no $t'\in N$ such that both $v_{1}$ and $v_{2}$  are marked by $t'.$ 
We know that each $C_{t'}$ has at most $|β(t')|\leq l$ vertices.
Therefore, if  $v_{1}$ (resp. $v_{2}$) is  marked by $t'_{1}$ (resp. $t'_{2}$)  and  $\mathsf{dist}_{D_{a}}(v_{1},v_{2})\geq  2l+1$
then $C_{t_{1}'}\cap C_{t_{2}'}=\emptyset$ which, in turn, implies that $v_{1}$ and $v_{2}$ are independent.

We now construct a set  $M'\subseteq M$ of  pairwise-independent strongly marked vertices of $D_{q}$
by greedily picking from $M$  
 a strongly  vertex $x$
and then excluding as possible choices all vertices of $M$ that are at  distance $<2l+1$ from $x$ in $D_{a}$.

Recall that $2l+1\leq  \widehat{h},$ therefore,
a vertex in the partially triangulated grid $D_{q}$
has at most $(1+4l)^2$ vertices at distance $\leq  2l$ in $D_{a}.$ We obtain that 
\begin{eqnarray}
    |M'|\geq  | M|/(1+4l)^2\geq (c'm^2+3)^4.\label{other_more_two}
\end{eqnarray}
Let $\ell\coloneqq c'm^2+3$ and recall that $\ell\leq  \widehat{h}.$

Taking \eqref{other_more_two} in mind, we have just found  a set of $\ell^4$ vertices of $M'$  that
belong to the central $(a-2\ell)\times(a-2\ell)$-grid of $D_{a}.$
By \cref{subjection}, $D_{a}$ contains an $(\ell\times\ell)$-grid $\Gamma_{\ell}$ as an $M'$-minor.
This also implies that $G^{\mathsf{c}}$ contains $\Gamma_{\ell}$ as a $M'$-minor.
 
 For every vertex $u\in \Gamma_{\ell},$ we know that the set of its precursors in  $G^{\mathsf{c}}$, contains some vertex $z_{u}$ of some set $C_{t'_{u}}$.
As there may be many possible choices of $z_{u},$ from now on, we arbitrarily fix one so to  uniquely correspond to every vertex $u$ of $\Gamma_{\ell}$, a vertex 
$t'_{u}\in N$ and a subgraph $C_{t'_{u}}$ of $G^{\mathsf{c}}$.

By the way we defined $M',$ it holds that, for distinct $u$'s in $\Gamma_{\ell},$ the corresponding  $C_{t'_{u}}$'s are {pairwise disjoint} and moreover, the corresponding 
$t'_{u}$'s as {well as the corresponding  $z_{u}$'s are also pairwise distinct}. 
Let $u\in V(\Gamma_{\ell}).$
As $G$, and therefore also $G^{\mathsf{c}}$,  is 4-connected, we have that $|A_{t'_u}|\geq  4$. Moreover, in $G^{\mathsf{c}}$, there are four internally disjoint paths in $G[β(t'_{u})]$ starting from $z_{u}\in C_{t'_{u}},$ finishing at any possible choice of four distinct vertices of $A_{t'_{u}}$ and such that the internal vertices of these paths all belong to $C_{t'_{u}}.$

\begin{figure}[h!]
\begin{center}
\includegraphics[height = 6.3cm]{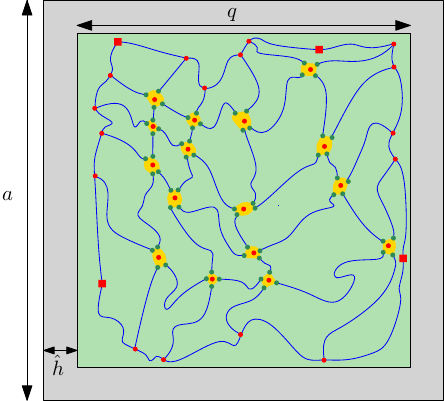}
\end{center}
  \caption{A visualization of the 
grids $D_{a}$ and $D_{q}$ in the proof of \cref{claimmarked}.
In \red{red} we depict the vertices $z_{t'_{u}}$ as they appear in the $(\ell\times\ell)$-grid $\Gamma_{\ell}$ (the corners are squares).
In \blue{blue} we depict the subdivided paths of $\Gamma_{\ell}$ that are also paths of $G^{\mathsf{c}}.$ The corresponding  $C_{t'_{u}}$'s (resp. $A_{t'}$'s) are in \textcolor{chromeyellow}{yellow} (resp. \darkgreen{green}).}
  \label{drunkenness}
\end{figure}

 The above imply that the $4$-connected graph $G^{\mathsf{c}}$ contains as a subgraph a subdivision $L$ of the $(\ell\times \ell)$-grid 
where the vertices of degree four are all vertices of 
$$Z\coloneqq \{z_{u} \mid u\in V(\Gamma_{\ell})\}$$ 
(see \cref{drunkenness}). 
Let $x\coloneqq z_{u}\in Z$ and assume that, in the subdivided grid $L,$ 
$x$ is the intersection of the ``horizontal'' path $H_{x}$ and the ``vertical'' path $V_{x}.$ 

We traverse $H_{x}$ from left to right and we define  $x^{\textrm{left}}$ to be the first vertex of $A_{t'_{u}}$ appearing on $H_{x}$ and  $x^{\textrm{right}}$ to be the last vertex of $A_{t'_{u}}$ appearing on $H_{x}.$
Similarly, we traverse $V_{x}$ from up to down and we define  $x^{\textrm{up}}$ to be the first vertex of $A_{t'_u}$ appearing on $V_{x}$ and  $x^{\textrm{down}}$ to be the last vertex of $A_{t'_u}$ appearing on $V_{x}.$
We now construct a subgraph $L'$ of $L$ by removing,  for each $x=z_{t'_{u}}\in Z,$ all  internal vertices of the path between $x^{\textrm{left}}$ and $x^{\textrm{right}}$ and all internal vertices of path between $x^{\textrm{up}}$ and $x^{\textrm{down}}.$ 

Notice, that by the choice of the quadruples $\{x^{\textrm{left}},x^{\textrm{right}},x^{\textrm{up}},x^{\textrm{down}}\},$ the remaining graph $L'$ is a subgraph of $G^{\mathsf{c}}.$
Moreover, if, for every $x=z_{t'_{u}}\in Z,$ we add all six pairs from $\{x^{\textrm{left}},x^{\textrm{right}},x^{\textrm{up}},x^{\textrm{down}}\}$ as edges to $L'$ then we have an $(\ell-2)$-crossed grid (depicted in \cref{liquidating}) as a subgraph of $G^{\mathsf{c}}$ (these six edges exist in $G^{\mathsf{c}}$ as they are edges of the clique induced by $A_{t'_{u}}$).
Recall that $\ell=c'm^2+3.$
From~\cref{commoditized} we obtain that $c'm^2+1=\ell-2\leq  c'\cdot (\hw(G^{\mathsf{c}}))^2\leq  ^{\eqref{precensored}}c'm^2,$ a contradiction and the claim follows.
\end{cproof}

\begin{figure}[h!]
\begin{center}
\includegraphics[height = 6.3cm]{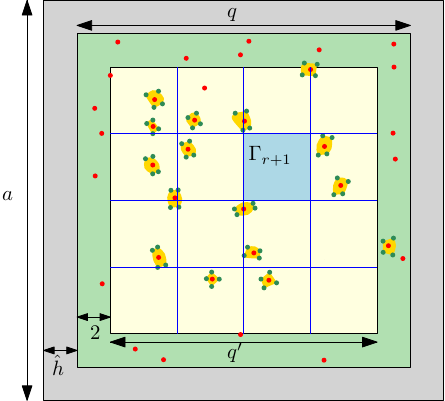}
\end{center}
  \caption{A visualization of the last part of the proof of \cref{blackbox}. The \blue{blue} lines indicate the packing of the $t+1$ pairwise vertex-disjoint $(r+1,r+1)$-grids in $D_{q'}$}
  \label{flattering}
\end{figure}

We set $$q'\coloneqq q-4=   
a-2\widehat{h}-4=\lceil{(r+1)^{1/2}}\rceil\cdot (t+1).$$
We also denote by $D_{q'}$ the central 
$(q' \times q')$-grid   of $D_{q}.$
Let $B^\star$ be the vertices of  $V(D_{q'})$ 
whose set of precursors in $G^{\mathsf{c}}$ contains a vertex of some $β(t'), t'\in N$. 

If $u\in B^\star$ and $u$ is not a marked vertex of $D_{q}$, then 
there is a precursor $x$  of $u$ that belongs in $A_{t'}$ for some $t'\in N$.
Moreover, as $u$ is not marked, none of the precursors of $u$ belongs in $C_{t'}$.
Notice now that as $C_{t'}$ is non-empty and   $G[\beta(t')]$ is connected, 
there is a vertex $x''\in C_{t'}$ that  has as a neighbor a vertex $x'$ of $A_{t'}$. 
The vertex $x'$ is, in turn, is either $x$ or is a neighbor of  
$x$ in the clique $G^{\mathsf{c}}[A_{t'}]$. This means that the distance of $x$ and $x''$ in $G^{\mathsf{c}}$
is at most two. 
As $x''$   is not a precursor of $u$, it should be one of the precursors of a marked vertex, say $v$ of $D_{q}$. This implies that every vertex $u\in B^\star$
should be within distance at most two from a marked vertex of   $D_{q}.$  As a consequence, 
$|B^\star|\leq  25d=t.$  

Notice that, by their definition, the  vertices in $V(D_{q'})\setminus B$ have precursors  in $H$ that  are all vertices of $G^{\mathsf{c}}-B.$

As $q'=\lceil{(t+1)^{1/2}}\rceil\cdot (r+1),$ we may consider a packing of 
$t+1$ pairwise vertex-disjoint $((r+1)\times (r+1))$-grids of $D_{q'}$.
Clearly,  for some  of these $((r+1)\times (r+1))$-grids, 
it holds that the precursors in $G^{\mathsf{c}}$ of all of its vertices 
are all vertices outside $B$ (see \cref{flattering}). This implies that $\bdim(G^{\mathsf{c}}-B,X\setminus B)>r$, therefore \eqref{eq_all_trages} holds, as required.
\end{proof}
We are now ready to prove the  main result of this section, that  is the following.

\begin{lemma}
\label{victorious}
Let $(\mathsf{h},\mathsf{c})\in \mathbb{N}\times[0,2]\setminus\{(0,0)\}$
and let $\Sigma\in\Sbbb_{\mathsf{h},\mathsf{c}}.$
Then $\bdim_{\Sigma}^{\star}(\mathscr{D}_{k}^{(\mathsf{h},\mathsf{c})})\in \Omega_{\mathsf{h},\mathsf{c}}(k^{1/480}).$
\end{lemma}

\begin{proof}
Let $η=\max\{\hw(G)\mid G\in\Ecal_{\Sigma}\}.$
Notice that $\widetilde{\mathscr{D}}_{k}^{(\mathsf{h},\mathsf{c})}\leq  \widehat{\mathscr{D}}_{k}^{(\mathsf{h},\mathsf{c})}\leq  \widetilde{\mathscr{D}}_{k+1}^{(\mathsf{h},\mathsf{c})},$ therefore% 
$\widetilde{\mathscr{D}}^{(\mathsf{h},\mathsf{c})}\sim_{{\mathsf{L}}} \widehat{\mathscr{D}}^{(\mathsf{h},\mathsf{c})}.$

By \cref{mediocrity}, we obtain that ${\mathscr{D}}^{(\mathsf{h},\mathsf{c})}\sim_{{\mathsf{L}}} \widehat{\mathscr{D}}^{(\mathsf{h},\mathsf{c})}.$  
By \cref{antisthenes} and the fact that ${\mathscr{D}}^{(\mathsf{h},\mathsf{c})}\sim_{{\mathsf{L}}} \widehat{\mathscr{D}}^{(\mathsf{h},\mathsf{c})},$ it is enough to prove that 

\begin{eqnarray}
\bdim_{\Sigma}(\widehat{\mathscr{D}}_{k}^{(\mathsf{h},\mathsf{c})})\in\mathcal{O}_{\mathsf{h},\mathsf{c}}\big(\bdim_{\Sigma}^{\star}(\widehat{\mathscr{D}}_{k}^{(\mathsf{h},\mathsf{c})})^{96}\big).\label{thidko_8iufo}
\end{eqnarray}
We prefer to work with $\widehat{\mathscr{D}}_{k}^{(\mathsf{h},\mathsf{c})}$ (depicted in \cref{tolerantly}), instead of $\widetilde{\mathscr{D}}_{k+1}^{(\mathsf{h},\mathsf{c})}$ or ${\mathscr{D}}_{k+1}^{(\mathsf{h},\mathsf{c})}$, because it is $4$-connected. 

For  simplicity, we set $G\coloneqq \widehat{\mathscr{D}}_{k}^{(\mathsf{h},\mathsf{c})}$, and we set $r\coloneqq \bdim_{\Sigma}^{\star}(\widehat{\mathscr{D}}_{k}^{(\mathsf{h},\mathsf{c})}),$ therefore there exists a tree decomposition ${(T,\beta)}$ of $G$ such that, for every $t\in V(T)$,
\begin{eqnarray}
\max\{{\bdim_{\Sigma}(G_{t})\mid t\in V(T)}\} & \leq   & r.\label{conservative}
\end{eqnarray}

Our objective is to prove that 
$\bdim_{\Sigma}(G)\in \mathcal{O}_{\mathsf{h},\mathsf{c}}(r^{96})$.
We set 
\begin{eqnarray}
m & = &  r^2+η+1\label{naturphilosophie}
\end{eqnarray} and we may assume that 
$2(2m+1)^{2}  <  4k^2$, otherwise 
$\bdim_{\Sigma}(G)\leq  |V(G)|=4k^2(h+c)\leq 2(2m+1)^{2}(h+c)$$\in \mathcal{O}_{\mathsf{h},\mathsf{c}}(r^4)\subseteq \mathcal{O}_{\mathsf{h},\mathsf{c}}(r^{96})$, and we are done.  This in turn implies that
\begin{eqnarray}
2(2m+1)^{2} & < & 4k^2({h}+c),  \mbox{~and} \label{assurances}\\
m & < & k \label{introjected}
\end{eqnarray}

Notice that the assumptions of \cref{letas_s} hold
for $\p=\bdim_{\Sigma}$ (for the function $h(x)=x^2+η+1$, because of \cref{neutrality})
and for  $\mathscr{H}=\widehat{\mathscr{D}}^{(\mathsf{h},\mathsf{c})}$ for  the function $g(x)=(2x+1)^2$ (because of  \cref{vegetables}).
Moreover,  given \eqref{conservative} and \eqref{naturphilosophie}, the conditions $2(g(h(r)))<|V(\mathscr{H}_k)|$
 and $h(r)<k$ are satisfied, because of \eqref{assurances} and \eqref{introjected}
 respectively. Therefore, $G$ has a tree decomposition $(\widehat{T},\widehat{\beta})$ where 
$\widehat{T}$ is a star with center $t$.
Let $G_{t}$ be the torso of $(\widehat{T},\widehat{\beta})$ on $t$. From  \cref{letas_s}, it  holds that $\bdim_{\Sigma}(G_{t})\leq r$, therefore there is some $X\subseteq V(G_{t})$, such that 
\begin{eqnarray}
\bdim_{\Sigma}(G_t,X)\leq r
\mbox{~and~}\\
G_t-X\in \Ecal_{\Sigma}.\label{ofkol_t89io}
\end{eqnarray}
Moreover, for every  $e=tt'\in E(\widehat{T})$, $\mathscr{H}_k[\widehat{β}(t')\setminus \widehat{β}(t)]$ is a connected graph and $|\widehat{β}(t')|\leq g(h(r))=(2m+1)^{2}$.
\medskip

For all $t'\in N\coloneqq N_{\widehat{T}}(t),$ we define $A_{t'}=\widehat{β}(t)\cap\widehat{β}(t')$, $B_{t'}=\widehat{β}(t')$,
and   $C_{t'}=B_{t'}\setminus A_{t'}$ and keep in mind that each $C_{t'}$ induces in $G$ a connected graph where
\begin{eqnarray}
|C_{t'}|\leq  |B_{t'}| \leq    (2m+1)^{2}.\label{discontinuities}
\end{eqnarray}
\smallskip
Let $G^{\mathsf{c}}\coloneqq \bigcup_{t\in V(T)}G_{t},$ $B\coloneqq \bigcup_{t'\in N}B_{t'},$ and $C\coloneqq\bigcup_{t'\in N}C_{t'}.$
By \cref{identifies} and \cref{neutrality}, we have
\begin{eqnarray}
\hw(G^{\mathsf{c}}) & \leq  & m.\label{precensored}
\end{eqnarray}

 Clearly $G$ is a (spanning) subgraph of $G^{\mathsf{c}}.$
 Instead of proving $\bdim_{\Sigma}(G)\in \mathcal{O}_{\mathsf{h},\mathsf{c}}(r^{96}),$ we will prove the stronger statement $\bdim_{\Sigma}(G^{\mathsf{c}})\in \mathcal{O}_{\mathsf{h},\mathsf{c}}(r^{96}).$ 
 For this, we observe that $G^\mathsf{c}-(X\cup B)$ is a subgraph of $G_{t}-X\in^{\eqref{ofkol_t89io}}\Ecal_{\Sigma},$ therefore the stronger statement will follow if we prove that {$\bdim(G^{\mathsf{c}},X\cup B) \in \mathcal{O}_{\mathsf{h},\mathsf{c}}(r^{96}).$}
 
 For this notice that \cref{blackbox} applies for $\Delta=4$, $(T,\beta)=(\widehat{T},\widehat{\beta})$, and
 $l=(2m+1)^{2}\in\Ocal(m^2)$.  As $m\in\Ocal_{\mathsf{h},\mathsf{c}}(r^2)$, \cref{blackbox} implies that $\bdim(G^c,X\cup B)\in\Ocal_{\mathsf{h},\mathsf{c}}(r^{96})$. As we proved \eqref{thidko_8iufo}, this completes the proof of the lemma.
\end{proof}

We are now in position to prove the lower bound in \cref{unrelatedly}.

\begin{corollary}
\label{predominant}
For every  non-empty finite and closed set of surfaces $\Sbbb,$
$\mathsf{sobs}(\Sbbb)\mbox{-}\bg\preceq_\mathsf{P}\Sbbb\mbox{-}\mathsf{\tw}.$
\end{corollary}

\begin{proof}
Recall that, $\mathsf{sobs}(\{\Sigma^{\varnothing}\})=\{\Sigma^{(0,0)}\}$ and,  by the Grid Theorem,  
$\{\Sigma^{\varnothing}\}\mbox{-}\tw\sim  \{\Sigma^{(0,0)}\}\mbox{-}\bg.$
Therefore, we may examine the case where $\mathsf{sobs}(\Sbbb)$
contains one or two surfaces $\Sigma^{(\mathsf{h}_{i},\mathsf{c}_{i})},i\in[t], t\in\{1,2\}$ different from the sphere: one orientable (assuming $h_{1}>0$ and $c_{1}=0$) and, possibly, one non-orientable, in case it also contains a non-orientable surface (assuming, that  $\mathsf{h}_{1}>\mathsf{h}_{2}\geq 0$ and $\mathsf{c}_{2}>0$).
Recall also that, by the definitions of $\Sbbb\mbox{-}\mathsf{\tw}$ in \eqref{attempting} and $\bdim_{\Sigma}^{\star}$ in \eqref{bidimsdefpre} and \eqref{bidimsdef}, we have that,  
for every graph $G,$ 
\begin{eqnarray}
\Sbbb\mbox{-}\mathsf{\tw}(G) & = & \min\{\bdim_{\Sigma}^{\star}\mid \Sigma\in\Sbbb\}.
\label{underworld}
\end{eqnarray}

Let $G$ be a graph where $\mathsf{sobs}(\Sbbb)\mbox{-}\bg(G)\geq k.$
From the definition of $\mathsf{sobs}(\Sbbb)\mbox{-}\bg$ in  \eqref{travestied}, this means that $G$ contains the graph $\mathscr{D}_{k}^{\Sigma^{(\mathsf{h}_{i},\mathsf{c}_{i}})}$ as a minor for some $i\in[t].$
Notice  that the fact that $\Sigma^{(\mathsf{h}_{i},\mathsf{c}_{i})}\in \mathsf{sobs}(\Sbbb)$ implies that $\Sbbb\subseteq\Sbbb_{h_i,c_i}$. 
Indeed, $\Sbbb\subseteq\Sbbb_{\mathsf{h}_i,\mathsf{c}_i}$ follows by \cref{territorial} and the definition of $\Sbbb_{\mathsf{h},\mathsf{c}}$ (that is the set containing every surface where it is not possible to embed all graphs in $\mathscr{D}^{(\mathsf{h},\mathsf{c})},$ defined in the beginning of \cref{overturning}).

By \cref{victorious}, for every $\Sigma\in\Sbbb_{\mathsf{h},\mathsf{c}}$ it holds that $\bdim_{\Sigma}^{\star}(\mathscr{D}_{k}^{(\mathsf{h},\mathsf{c})})\in \Omega(k^{1/480}),$  by \eqref{underworld} and \cref{objectivity}, implies that  $\Sbbb\mbox{-}\mathsf{\tw}(G)\in \Omega(k^{1/480})$ and we are done.\end{proof}

\section{Putting everything together}\label{beneficial}
We define a more refined version of 
${\Sbbb}\mbox{-}\mathsf{\hw},$ namely ${\Sbbb}\mbox{-}\mathsf{\mathsf{rhw}},$ 
where:
\begin{quote}
 ${\Sbbb}\mbox{-}\mathsf{\mathsf{rhw}}(G)$
is the minimum $k$ for which   $G$ has a tree decomposition $(T,\beta)$ of adhesion  at most $k$ such that  for every $d\in V(T),$
\begin{itemize}
\item either $|\beta(d)|\leq  k,$ or 
\item there exist a set $A_d\subseteq\beta(d)$ of size at most $k$ and a surface $\Sigma\in\Sbbb$ such that $G_d-A_d$ has a $\Sigma$-decomposition $\delta $ of width at most $k$ and breadth at most $k,$ and 

\!\!\!\!\!\!\!${\blacktriangleright}$ if $X_d$ is the set of all vertices of $G_d-A_d$ drawn in the interior of the vortices of $\delta$ together with the vertices from $A_{d}$, then
    \begin{enumerate}
        \item there is no vertex of $G_d-A_d$ which is drawn in the interior of a non-vortex cell of $\delta$,
        \item $G_d-X_d$ is a minor of $G$, and
        \item for every  {neighbor} $d'$ of $d$ in $T$ it holds that $|\big(\beta(d)\cap\beta(d')\big)\setminus\big(A_d\cup X_d\big)|\leq 3$.
    \end{enumerate}
\end{itemize}
\end{quote}

Notice that the main difference of the above definition with the one of $\Sbbb\mbox{-}\hw$ in  \eqref{eradication} is the imposition of the bound to the adhesion and the extra conditions 1, 2, and 3. The first condition
says that, outside the vortices, the $\Sigma$-decomposition $\delta$ outside its vortices is just an embedding.
The second condition says  that the ``surface part'' of each adhesion should not have more than 3 vertices, and the
third condition says that the ``surface part'' part of each 
torso is a minor of the graph.
We stress that these two conditions are important for algorithmic applications  of the GMST. 

Our next result implies that
by removing from the above definition the 
demand for a bound adhesion or the three extra conditions 
the defined parameter is the same (up to parameter equivalence).

We now give the following most extended version of \cref{annihilation}.
\begin{theorem}
  \label{abominations}
  For every finite and  closed set of surfaces $\Sbbb,$ it holds that the parameters ${\Sbbb}\mbox{-}\mathsf{\tw},$ ${\Sbbb}\mbox{-}\mathsf{\hw},$ ${\Sbbb}\mbox{-}\mathsf{\mathsf{rhw}},$ and ${\mathsf{sobs}(\Sbbb)}\mbox{-}\bg$ are equivalent, i.e., $\Sbbb\mbox{-}\mathsf{\tw}\sim{\Sbbb\mbox{-}\mathsf{\hw}}\sim{\Sbbb}\mbox{-}\mathsf{\mathsf{rhw}}\sim {\mathsf{sobs}(\Sbbb)\mbox{-}\bg}.$
\end{theorem}

\begin{proof}
Recall that $\mathsf{sobs}(\emptyset)=\{\Sigma^{\varnothing}\}$
and, by definition, all values of 
$\emptyset\mbox{-}\mathsf{\tw},$
$\emptyset\mbox{-}\mathsf{\hw},$
$\emptyset\mbox{-}\mathsf{\mathsf{rhw}},$ and 
$\{\Sigma^{\varnothing}\}\mbox{-}\bg$ are equal to $\infty.$ Therefore,  we may assume that $\Sbbb$ is non-empty.
We prove  (in order) 
\begin{eqnarray}
{\mathsf{sobs}(\Sbbb)\mbox{-}\bg}\preceq_{\mathsf{P}} \Sbbb\mbox{-}\mathsf{\tw}\preceq_{\mathsf{P}} \Sbbb\mbox{-}\mathsf{\hw}\preceq_{\mathsf{L}} \Sbbb\mbox{-}\mathsf{\mathsf{rhw}}\preceq {\mathsf{sobs}(\Sbbb)\mbox{-}\bg}.\label{final_meop}
\end{eqnarray}

\noindent \emph{Proof of} ${\mathsf{sobs}(\Sbbb)\mbox{-}\bg}\preceq_{\sf P} \Sbbb\mbox{-}\mathsf{\tw}.$ This is \cref{predominant}.

\noindent \emph{Proof of} $\Sbbb\mbox{-}\mathsf{\tw}\preceq_{\mathsf{P}} \Sbbb\mbox{-}\mathsf{\text{$\hw$}}.$ This follows from \cref{tiko_news}.

\noindent \emph{Proof of}  $\Sbbb\mbox{-}\mathsf{\hw}\preceq_{\mathsf{L}} \Sbbb\mbox{-}\mathsf{\mathsf{rhw}}.$ This follows directly because the definition of $\Sbbb\mbox{-}\mathsf{\hw}$ is the definition of $\Sbbb\mbox{-}\mathsf{\mathsf{rhw}}$ without the adhesion bound and the additional conditions 1, 2, and 3.

\noindent \emph{Proof of} $\Sbbb\mbox{-}\mathsf{\mathsf{rhw}}\preceq {\mathsf{sobs}(\Sbbb)\mbox{-}\bg}.$ Follows directly from \cref{thm_globalstructure}.
\end{proof}

Let us now present  the general statement of the main structural result of the Graph Minors Series. 

\begin{proposition}[Graph Minors Structure Theorem~\cite{robertson2003graph,KawarabayashiTW20Quicklyexcluding}]\label{intimation} 
 There exists a function $f:\mathbb{N}\to\mathbb{N}$ such that 
for every graph $G,$ if $\hw(G)\leq  k,$ then  $G$ has a tree decomposition~${(T,\beta)}$ 
of adhesion $≤k$ such that for all~${t \in V(T)},$ 
there exist a set $A_t\subseteq\beta(t)$ of size at most $f(k)$  such that 
there exists a surface of Euler-genus at most $f(k)$ 
and a $\Sigma$-decomposition $\delta $ of $G_{d}-A_d$ with breadth at most~$f(k)$ and width at most~$f(k).$
\end{proposition}

The above already implies that  $\hw''\preceq\hw,$ where 
$\hw''$ is the parameter defined in \eqref{differences} (\cref{labyrinths}).

Let us now finish this section by going back in the parameter
$\hw'$ defined  in the introduction (\eqref{stereotypically} in \cref{recommendation}). Given \cref{intimation},
it appears that \cref{abstraction} is the only essential 
ingredient for proving the asymptotic equivalence of $\hw'$ with the Hadwiger number $\hw.$

\begin{theorem}
\label{retribution}
$\hw\sim\hw'\sim\hw''.$
\end{theorem}

\begin{proof}
As we already mentioned, $\hw''\preceq\hw$  follows from  \cref{intimation}.
The inequality $\hw'\preceq_{\mathsf{P}}\hw''$ follows similarly to \cref{tiko_news}, by seeing (using \cref{abstraction}) the set $A_{d}$ along with the vertices of  the vortices of $\delta $ as a low bidimensionality  modulator.
Finally, $\hw\preceq_{\mathsf{P}}\hw'$  follows because, for every $X\subseteq V(K_{k}),$ it holds that $\bdim(G,X)\geq  \lfloor \sqrt{|X|}\rfloor\in \Omega(\sqrt{k})$ and that $\eg(K_{k-|X|})\in \Omega(\sqrt{k-|X|}).$ 
\end{proof}

The above is  a restatement of \cref{main_mainl} that we presented in the introduction section.

\section{Discussion}\label{predecessors}

In this paper we defined a hierarchy of parameters that bijectively reflects the containment-relation  hierarchy of all two-dimensional surfaces without boundaries. 
This permitted us to reveal the  structure of graphs excluding surface embeddable graphs.
This structure is determined by tree decompositions 
whose torsos contain some set of low bidimensionality modulators to the embeddability 
to a surface where the excluded graph cannot be embedded. Also, our results imply that 
this decomposition characterization is 
optimal in the sense that there are graphs
embeddable in the same surface as the excluded graph where no such decomposition is possible.

\paragraph{Annotated extensions of graph parameters.}
Notice that our notion of bidimensionality, at its very core, extends the parameter ``biggest grid'', which asks for the
biggest $k$ such that $G$ contains $\Gamma_{k}$ as a minor, to annotated graphs.
That is, instead of asking for the largest grid in all of $G$, we now ask for the largest grid \textsl{rooted} at the annotation vertices of an annotated graph.
We stress here, that this idea can be put into a more general context as follows.

An \emph{annotated graph} is a pair $(G,X),$ where $X\subseteq V(G)$. We also denote by $\Acal_\mathsf{all}$ the class of all annotated graphs.
Given some minor-monotone graph parameter $\p:\Gall\to\Nbbb$, we define the 
\emph{annotated extension} of $\p$ as the 
annotated graph parameter $\p:\Acal_\mathsf{all}\to\Nbbb$ where for every $(G,X)\in\Acal_\mathsf{all}$,  
$$\p(G,X)=\max\{\p(H)\mid \text{$H$ is an $X$-minor of $G$}\}$$
Using the terminology above and the Grid Theorem, we 
observe that the annotated extension 
of treewidth is equivalent to the bidimensionality, i.e., there is some function $f:\Nbbb\to\Nbbb$ such that,  
for every annotated 
graph $(G,X)$, it holds that $\bdim(G,X)\leq \tw(G,X)≤f(\bdim(G,X))$. Moreover, because of the result of Chuzhoy and Tan in \cite{chuzhoy2021towards}, $f$ is a polynomial function. That way one may see bidimensionality as the annotated analogue of treewidth (with a polynomial functional gap).

\medskip
We should stress that the term ``bidimensionality'' has already been used in the context of graph algorithms as a  property of parametrized problems on graphs, introduced in \cite{DeFoHaTh2005}.
Such a problem is called \emph{bidimensional} if its \textsf{yes} (or \textsf{no})
instances are minor-closed (or contraction closed)
and if its \textsf{yes} (or \textsf{no}) instances exclude a $(O(k)\times O(k))$-grid as a minor (see also \cite{DeHa2005,DeHa2008theb,DemaineFHT04bidim,FoLoRaSa2011,FominLST20,FominGT11contr,DeHaTh06,FominGT11contr,FoLoSa2012,BasteT22contr} for related results).
In this paper we use the term ``bidimensionality'' as a parameter on annotated graphs. We believe that both uses of the term are justified as they refer to 
the way some set is spread in a 2-dimensional area. Also the newly introduced use is in a sense a descendant of the previous one: the main combinatorial result where our proofs are based, i.e.,  \cref{subjection},  originates from \cite{DemaineFHT04bidim}, where the bidimensionality of parameterized problems has been studied.

\paragraph{Bidimensionality and MSOL.} Recently bidimensionality provided an important ingredient for the statement of algorithmic meta-theorems. 
In \cite{SauST25}, a fragment of Monadic Second Order Logic (\textsl{MSOL})
has been introduced, namely $\mathsf{MSOL}/\tw$ by restricting set quantification
to sets of bounded bidimensionality.
According to the main result of  \cite{SauST25},
model checking for such formulas can be done in quadratic time on classes of graphs excluding some clique as a minor.

\paragraph{Torso-extensions vs annotated extensions.} 
Given an annotated graph $(G,X),$ we define $\mathsf{torso}(G,X)$ as the graph obtained by adding, for each connected component $C$ of $G-X,$ all edges between the vertices of $N_{G}(V(C))$ in $G[X]$. 
An alternative  way to extend a minor-monotone parameter $\p:\gall\to\Nbbb$  to annotated graphs  is to define the \emph{torso-extension} $\p^{\sf t}:\mathcal{A}_{\mathsf{all}}\to\Nbbb$ of $\p$   so that  $$\p^{\textsf{t}}(G,X)=\min\{\p(\textsf{torso}(G,Z))\mid X\subseteq Z\subseteq V(G)\}.$$ 
 This extension was suggested  by  Eiben, Ganian, Hamm, and Kwon \cite{EibenGHK21} and 
 had  interesting algorithmic applications in \cite{JansenK021verte,AgrawalKLPRSZ22delet,EibenGHK21}.
Notice that every  $X$-rooted minor of $G$ is a minor of $\torso(G,X)$ and this readily implies that  
$\p(G,X)\leq  \p^{\textsf{t}}(G,X).$ One may ask whether the other direction also holds, i.e., whether $\p$ and $\p^{\mathsf{t}}$ are equivalent for certain instantiations of $\p$. 
Clearly, in  case of treewidth, $\tw(G,X)$ (that is the bidimensionality of $X$ in $G$) can be much smaller than $\tw^{\textsf{t}}(G,X)$: denote by $\Gamma_{k}$   the $(k\times k)$-grid and  $X$  the vertices of its perimeter. Then, $\tw(\Gamma_{k},X)\in \mathcal{O}(1),$ while $\tw^{\mathsf{t}}(\Gamma_{k},X)\in \Omega(k^2).$
\\ 

For recent advances on structural theorems on the torso and the annotated extension of treewidth,  see \cite{protopapas2026colorfulminors}. It is an open problem for which parameters $\p$
the annotated extension of $\p$ is equivalent to the 
torso extension of $\p$.

\paragraph{Clique-sum extension.}
Another ingredient of structural theorems is the clique-sum closure. Given a graph parameter $\p\colon\mathcal{G}_{\text{all}}\to\mathbb{N}$ we define its \emph{clique-sum extension} as the 
graph parameter $\p^{\star}\colon\mathcal{G}_{\text{all}}\to\mathbb{N}$ where, for every graph $G,$
$\p^{\star}(G)$ is the minimum $k$ such that $G$ is in the clique-sum closure 
of the class of graphs where the value of $\p$ is at most $k.$ Notice that $\p^{\star}=(\p^{\star})^{\star}.$
It is easy to see that if $\textsf{size}$ is the parameter returning the number of vertices of a graph, then $\tw\sim_{\mathsf{L}}\textsf{size}^{\star}.$
Also it is easy to see that if $\p$ is a minor-monotone graph parameter, then  $\p^{\star}$ is also minor-monotone. Finally it also holds that if $\p$ 
is a graph parameter where $\tw\preceq \p\preceq \size$, then $\p^{\star}\sim\tw$.

Given a graph parameter $\p\colon\mathcal{G}_{\text{all}}\to\mathbb{N}$ and a minor-closed graph class $\mathcal{G}$ we define the \emph{modulator composition} of $\p$ and $\mathcal{G}$, as the parameter $\p_\mathcal{G}:\mathcal{G}_{\text{all}}\to\mathbb{N}$, so that 
\begin{eqnarray}
\p_\mathcal{G}(G) & \coloneqq & \min\{k\mid \mbox{$ \mbox{there exists an $X\subseteq V(G)$ such that~} \p(G,X)\leq  k\mbox{~and~}G-X\in \mathcal{G}$}\}.\label{resemblance}
\end{eqnarray} 
Notice that, given that $\p$ is minor-monotone and $\Gcal$ is minor closed, we have that $\p_\Gcal$ is minor-monotone as well.
By the above definitions, because of the result of Chuzhoy {and Tan}, it follows  that  $\tw_{\Ecal_\Sigma}\sim_{\mathsf{P}}\bdim_{\Sigma}$ and moreover that
\begin{eqnarray}
\tw^{\star}_{\Ecal_\Sigma} & \sim_{\mathsf{P}}  & \bdim^{\star}_{\Sigma}\label{foundastion}
\end{eqnarray}
\eqref{foundastion} implies that, in the definition of our treewidth extension in \eqref{attempting}, our results would be the same  if we 
write instead:
\begin{eqnarray} 
{\Sbbb}\mbox{-}\tw(G)\coloneqq \tw^{\star}_{\Ecal_{\Sbbb}}(G). \nonumber
\end{eqnarray}
It would be an interesting project to investigate the behavior, in particular with respect to universal obstructions, of the parameter $\p^{\star}_{\Ecal_{\Sbbb}}(G)$ for minor-monotone parameters $\p$ where $\tw\preceq\p$.

\paragraph{Parametric dependencies.}
Notice that in the proof of \cref{abominations} the only relation in \eqref{final_meop} that is not polynomial is $\Sbbb\mbox{-}\mathsf{\mathsf{rhw}}\preceq \mathsf{sobs}(\Sbbb)\mbox{-}\bg$, which is conditioned by the exponential bounds in \cref{thm_globalstructure}. Also the dependencies of the bounds 
of the GMST \cref{intimation} are super-polynomial; in fact they are of  order $2^{\poly(k)}$, according to \cite{KawarabayashiTW18anewp}.
To improve these dependencies to polynomial ones has been a running challenge in the field of structural graph theory
 for a long time.
Very recently this problem has been resolved affirmatively by Gorsky, Seweryn, and Wiederrecht \cite{GorskySW25Polynomial}.
In light of this result it seems reasonable to expect that also the final relation,  i.\@ e.\@~        $\Sbbb\mbox{-}\mathsf{\mathsf{rhw}}\preceq \mathsf{sobs}(\Sbbb)\mbox{-}\bg,$ can become polynomial.
This indicates that  \cref{abominations} should also hold with polynomial bounds.

\paragraph{Acknowledgements:} We are thankful to Laure Morelle, Christophe Paul, Evangelos Protopapas,  and Giannos Stamoulis, for their helpful comments on the containments and the presentation of this paper.

\end{document}